\documentclass[11pt]{amsart} 

\usepackage{stmaryrd}
\usepackage{tikz-cd}
\usepackage[color=orange!10]{todonotes}
\usepackage{paralist}
\usepackage{amsthm}
\usepackage[british]{babel}
\usepackage{csquotes}
\usepackage{enumitem}
\usepackage{mathtools, mathdots}

\usepackage[width=5.65in,height=8.0in,centering]{geometry}
\definecolor{chianti}{rgb}{0.6,0,0}
\definecolor{meretale}{rgb}{0,0,.6}
\definecolor{leaf}{rgb}{0,.35,0}
\usepackage[colorlinks,plainpages,backref,urlcolor=meretale,citecolor=meretale,linkcolor=chianti]{hyperref}

\usepackage{tikz}
\usetikzlibrary{arrows,cd}

\usepackage{mathrsfs}
\newcommand{\mathds}{\mathbb}

\usepackage{amsmath,amssymb,amsthm}
\usepackage{ytableau}
\usepackage{thmtools,thm-restate,colonequals}
\theoremstyle{definition}

\newcommand{\myendsymbol}{\ensuremath{\diamondsuit}}

\numberwithin{equation}{section}

\declaretheorem[
  style=definition,
  title=Example,
  qed={$\myendsymbol$},
  refname={example,examples},
  Refname={Example,Examples},
  sharenumber=bla,
]{exa}

\declaretheorem[
  style=definition,
  title=Definition,
  qed={$\myendsymbol$},
  sharenumber=bla,
]{dfn}

\declaretheorem[
  style=definition,
  title=Remark,
  qed={$\myendsymbol$},
  sharenumber=bla,
]{rmk}

\declaretheorem[
  style=plain,
  title=Lemma,
  qed={},
  sharenumber=bla,
]{lem}

\declaretheorem[
  style=plain,
  title=Theorem,
  qed={},
  sharenumber=bla,
]{thmx}

\declaretheorem[
  style=plain,
  title=Corollary,
  qed={},
  sharenumber=bla,
]{cor}

\declaretheorem[
  style=plain,
  title=Proposition,
  sharenumber=bla,
]{prp}

\declaretheorem[
  style=definition,
  title=Question,
  qed={$\myendsymbol$},
  sharenumber=bla,
]{question}

\declaretheorem[
  style=definition,
  title=Setup,
  qed={$\myendsymbol$},
  sharenumber=bla,
]{setup}

\declaretheorem[
  style=plain,
  title=Theorem,
]{maintheorem}

\declaretheorem[
  style=definition,
  title=Question,
  qed={$\myendsymbol$},
  sharenumber=maintheorem,
]{mainquestion}

\makeatletter
\def\tagform@#1{\maketag@@@{(\ignorespaces\textbf{#1}\unskip\@@italiccorr)}}
\renewcommand{\eqref}[1]{\textup{{\normalfont(\ref{#1}}\normalfont)}}
\makeatother

\DeclareMathOperator{\codim}{codim}   
\DeclareMathOperator{\rank}{rank}     
\DeclareMathOperator{\cptdim}{cptdim} 
\DeclareMathOperator{\RI}{RI} 
\DeclareMathOperator{\ara}{ara} 
\DeclareMathOperator{\lcd}{lcd} 
\DeclareMathOperator{\height}{ht} 
\DeclareMathOperator{\init}{in} 
\DeclareMathOperator{\sign}{sign} 
\DeclareMathOperator{\chr}{char} 

\DeclareMathOperator{\grade}{grade}
\DeclareMathOperator{\Fitt}{Fitt} 
\DeclareMathOperator{\tr}{tr} 

\newcommand{\into}{\hookrightarrow}

\newcommand{\lto}{\longrightarrow}

\newcommand{\Hc}{H_{\text{c,\'et}}}

\newcommand{\KK}{\mathds{K}}     
\newcommand{\NN}{\mathds{N}}     
\newcommand{\QQ}{\mathds{Q}}     
\newcommand{\ZZ}{{\mathds{Z}}}   

\newcommand{\fraka}{{\mathfrak a}}

\newcommand{\frakm}{{\mathfrak m}}
\newcommand{\frakp}{{\mathfrak p}}

\newcommand{\et}{\mathrm{\acute{e}t}}

\tikzset{square matrix/.style={
    matrix of nodes,
    column sep=-\pgflinewidth, row sep=-\pgflinewidth,
    nodes={draw,
      minimum height=0.5cm,
      anchor=center,
      text width=0.6cm,
      align=center,
      inner sep=0pt
    },
  },
  square matrix/.default=0.55cm
}

\title[arithmetic rank and cohomological dimension of residual intersections]{On the arithmetic rank and cohomological dimension of some residual intersections}

\author[Batavia]{Manav Batavia}
\address{Department of Mathematics, Purdue University, West Lafayette, IN~47907, USA}
\email{\href{mailto:mbatavia@purdue.edu}{mbatavia@purdue.edu}}
\thanks{}

\author[Pandey]{Vaibhav Pandey}
\address{Department of Mathematics, Indian Institute of Technology Madras, Chennai, Tamil Nadu~600036, India}
\email{\href{mailto:vaibhavpandey@iitm.ac.in}{vaibhavpandey@iitm.ac.in}}
\thanks{}

\author[Walther]{Uli Walther}
\address{Department of Mathematics, Purdue University, West Lafayette, IN~47907, USA}
\email{\href{mailto:walther@purdue.edu}{walther@purdue.edu}}
\thanks{}



\begin{document}
\begin{abstract}
We determine the arithmetic rank and cohomological dimension of each generic residual intersection of the ideal of submaximal Pfaffians of an alternating matrix of variables. This ideal fails to be a set-theoretic complete intersection. Our methods yield several structural results: in characteristic zero, we compute the $D$-module length of the top local cohomology module supported at this ideal and show that over the integers, it admits a $\QQ$-vector space structure. In particular, the top local cohomology module has no integer torsion. We also establish the $F$-regularity of the generic residual intersection rings considered here.

Further, for each admissible height $s$, we construct $s$-residual intersections of complete intersection ideals and height three Gorenstein ideals that are set-theoretic complete intersections in any characteristic. Finally, we pose a general question of independent interest concerning the behaviour of arithmetic rank.
\end{abstract}
\maketitle
\tableofcontents

\section{Introduction}

The \emph{arithmetic rank} of an ideal $I$ of a Noetherian ring $R$ is the least number of elements required to generate it up to radical, that is, 
\[
\ara(I)\colonequals \min\{k\geq 0:\text{there exist $f_{1},\cdots,f_{k}\in R$ such that $\sqrt{(f_{1}, \ldots, f_{k})} = \sqrt{I}$}\}.
\]
Note that when $R$ is a polynomial ring in finitely many variables over an algebraically closed field, $\ara(I)$ is the minimal number of equations needed to cut out the algebraic set $V(I)=V(\sqrt{I})$. The ideal $I$ is said to be a \emph{set-theoretic complete intersection} if $\ara(I)=\height(I)$.

Arithmetic rank is a notoriously difficult algebro-geometric invariant to compute. Its determination has motivated a long line of work including the foundational contributions of Cowsik--Nori \cite{CowsikNori} and Hartshorne \cite{HartshorneCI}. Subsequent computations of arithmetic rank for several families of ideals such as \cite{LyubeznikBAMS,BrunsSchwanzl,Barile,BarileLyubeznik} have illustrated
the rich interplay between constructing explicit set-theoretic generators, combinatorial connections, and cohomological obstructions.

Local cohomology provides a fundamental lower bound for arithmetic rank. Recall that the \emph{(local) cohomological dimension} of $I$ in $R$ is 
\[
\lcd(I)\colonequals\max\{n\:|\:H^n_I(R)\neq 0\}.
\]
We always have 
\[\height(I) \leq \lcd(I) \leq \ara(I).\]

We compute the arithmetic rank of generic residual intersections of height three Gorenstein ideals and study their top nonvanishing local cohomology modules. Residual intersections originated in classical enumerative geometry, notably in Chasles' 1864 solution to Steiner's problem of counting plane conics tangent to five general plane conics \cite{Chasles}. The modern algebraic theory of residual intersections began with Artin and Nagata \cite{Artin-Nagata}. Central to residual intersection theory is computing the generators of residual intersections and determining if it is Cohen--Macaulay. These questions have motivated an extensive body of research, see \cite{HunekeRI,HunekeUlrichRI,
HunekeUlrichGenericRI,UlrichANproperties,CEU,CNT,
EisenbudUlrichRI}. Within this theory, the residual intersections of height three Gorenstein ideals are particularly well-understood. The generators of these residual intersections have been computed explicitly, which we recall next. This provides the motivation for our computation of their set-theoretic generators.    

Let $\KK$ be a field or the integers $\ZZ$; let $n\geq2$ and $s\geq3$. Let $X$ and $Y$ be generic alternating $n\times n$ and generic $n\times s$ matrices, respectively. Set $R=\KK[X,Y]$ and
\[
M=\begin{pmatrix}
X&Y\\
-Y^{\tr} &0
\end{pmatrix},
\]
where $0$ indicates the zero matrix of size $s \times s$. Define $P(n,s)$ to be the ideal generated by the Pfaffians of the principal submatrices of $M$ containing $X$. 

When $n$ is odd, the signed submaximal Pfaffians $p_1,\ldots,p_n$ of $X$ generate a generic Gorenstein ideal $I$ of height three \cite[Theorem 2.1]{BuchsbaumEisenbudAJM}. Kustin and Ulrich \cite[Theorem 4.7]{KustinUlrich} proved that the $s$-residual intersection of $I$ is precisely $P(n,s)$, that is,
\[
\bigl((p_1,\ldots,p_n)Y\bigr):IR=P(n,s).
\]
Observe that when $n$ is even, the generators of $P(n,s)$ still make sense, though in this case this ideal is not a residual intersection. Our computations for odd $n$ require us to work with the larger family $P(n,s)$ allowing both parities, since our arguments pass between even and odd values of $n$. 

If $n$ is odd, by \cite[Theorem 3.3]{HunekeUlrichRI}, $P(n,s)$ is a prime ideal (of height $s$) and $R/P(n,s)$ is a Cohen--Macaulay ring by \cite[Theorem 1.4(b)]{HunekeUlrichGenericRI}. It follows from a well-known result of \cite{PeskineSzpiro} that $\lcd(P(n,s))=s$ if $\KK$ has positive characteristic. The free resolution of $P(n,s)$ is worked out in \cite{Kustin-Ulrich-Memoirs}. 

Our first main result determines the arithmetic rank and cohomological dimension of the ideal $P(n,s)$.

\begin{maintheorem}\label{mainthm:A}(cf. Theorems \ref{thm:cdP(n,s)} and \ref{thm:ara-in-char-p})
For all $s\geq 3$ and $n\geq 2$, we have
\[
\ara(P(n,s))=\binom{s}{2}-\binom{s-n}{2}+1.
\] 
Furthermore, if $\KK$ has characteristic $0$, then $\lcd(P(n,s))=\ara(P(n,s))$.
\end{maintheorem}

Our computation extends the study of generic residual intersections of complete intersection ideals in \cite{BMSMP}. For the upper bound, we equip the set $\mathcal P$ of Pfaffian generators with a partial order and prove that the auxiliary algebra $\KK[\mathcal P]$ is a\textit{ symmetric algebra with straightening law (ASL)}. This gives the explicit set-theoretic generators of $P(n,s)$ (see Corollary \ref{cor:dimauxring}).

The (cohomological) lower bound requires a different approach from \cite{BMSMP}. Rather than deriving it from invariant theory, we carefully analyse the Pfaffian generators of $P(n,s)$. Two key observations form the backbone of our computations:
\begin{itemize}[leftmargin=*, align=left]
    \item The ideal $P(n,s)\cap(X)$ is generated up to radical by a subset of $\mathcal P$ that forms an ideal of the poset structure on $\mathcal P$ with the reversed order (Lemma \ref{lem:PintX}).
    \item Localization at an entry of $Y$
    identifies $P(n,s)$ with an extension of $P(n-1,s)$
    (Lemma \ref{lem:localization}), thus interlacing the cases of even and odd $n$.
\end{itemize}

The usefulness of Lemma \ref{lem:localization} goes beyond our computations of arithmetic rank and cohomological dimension. We use it to prove that if $\KK$ is an $F$-finite field of positive characteristic, then the generic $s$-residual intersection ring $R/P(n,s)$ is strongly $F$-regular (Theorem \ref{thm:P(n,s)Freg}). This extends \cite[Theorem 4.3]{Pandey} where the $F$-regularity is established for the\textit{ generic link}, that is, the case $s=3$. In addition, we prove structural results about the top nonvanishing local cohomology supported at $P(n,s)$. Before stating our second main result, we recall that quite generally, a local cohomology module supported at an ideal of a polynomial ring $R$ over a field of characteristic zero is known to be a \textit{holonomic} module over the Weyl algebra $D$, and hence has finite length as a $D$-module. We prove:

\begin{maintheorem}\label{mainthm:B}(cf. Theorem \ref{thm:Dmodlen} and Corollary \ref{cor:integertorsion})
Set $P=P(n,s)$ in $R=\KK[X,Y]$. If $\KK$ is a field of characteristic zero, the $D$-module length of the local cohomology module $H^{\ara(P)}_P(R)$ is as follows: 
\[
\ell_D(H^{\ara(P)}_P(R)) = 
\begin{cases} 
1 & \text{if } s > n, \\ 
2 & \text{if } s = n. 
\end{cases}
\]
If $\KK=\ZZ$, then $H^{\ara(P)}_P(R)$ is a $\QQ$-vector space; in particular, it has no integer torsion.
\end{maintheorem}

Theorem \ref{mainthm:A} and \cite[Theorem 4.7]{BMSMP} show that the generic residual intersections of generic height three Gorenstein ideals and generic complete intersection ideals (of height at least two) are \textit{not} set-theoretic complete intersections. In Section \ref{sec:STCI}, we construct linear specializations of these generic residual intersections that are set-theoretic complete intersections, while preserving the height and the residual intersection property.

\begin{maintheorem}\label{mainthm:C}(cf. Theorems \ref{thm:STCIforBMSMP} and \ref{thm:STCIforGor})
Let $n\geq 2$.
\begin{enumerate}[leftmargin=*, align=left, label=(\alph*)]
    \item Let $I$ be an ideal generated by a regular sequence of $n$ elements in a Gorenstein local ring $A$. For every $s\geq n$, there is a polynomial extension $B$ of $A$ in which $IB$ admits an $s$-residual intersection that is a set-theoretic complete intersection.

    \item Let $\KK$ be a field, $n$ be odd, and $s\geq3$. Let $I$ be the ideal of submaximal Pfaffians of a generic alternating $n\times n$ matrix. Its generic $s$-residual intersection $P(n,s)$ admits a linear specialization to a set-theoretic complete intersection $P'$ that remains an $s$-residual intersection of a specialization $I'$ of $I$. Moreover, $I'$ is Gorenstein of height three with $\mu(I')=\mu(I)=n$.
\end{enumerate}
\end{maintheorem}

As noted in Corollary \ref{cor:araRI}, Theorem \ref{mainthm:A} produces a sharp upper bound for the arithmetic rank of any residual intersection of a height three perfect Gorenstein ideal, while the constructions in Theorem \ref{mainthm:C} show that the height lower bound is attained. The analogous conclusions hold for complete intersection ideals by \cite[Theorem 4.9]{BMSMP} and Theorem \ref{mainthm:C}.

Finally, in Section \ref{sec:questions}, we ask a natural question on the behaviour of arithmetic rank inspired by the nuances of our arithmetic rank computation in Theorem \ref{mainthm:A}. 
\begin{mainquestion}\label{que:mainqnD}
Let $\KK$ be a field, $R=\KK[x_1,\ldots,x_t]$, and $I$ a proper ideal of $R$. Let $z$ be a new variable; set $S=R[z]$. Is
\[\ara(IS+(z)) = \ara(I)+1?\qedhere\]
\end{mainquestion}
While Question \ref{que:mainqnD} is open in general, we establish affirmative answers in some cases and produce counterexamples when its hypotheses are weakened. We also formulate an equivalent version---Question \ref{question:alternate}---which also provides a method for computing cohomological dimension in certain examples (Examples \ref{exa:hypersurface} and \ref{exa:general}).

\subsection*{Notation}
All rings considered are Noetherian and unital. If $I$ is an ideal of the ring $R$, we write $\mu(I)$ for the minimal number of elements needed to generate $I$. The binomial coefficient $\binom{a}{b}$ is zero if $a<0$.  

\section{Preliminaries}

\subsection{Residual intersections}

Let $X$ and $Y$ be two irreducible closed subschemes of a Noetherian scheme $Z$ with $\codim_Z X \leq \codim_Z Y =s$ and $Y \nsubseteq X$. Then $Y$ is called a \emph{residual intersection of $X$} if the number of equations needed to define $X\cup Y$ as a subscheme of $Z$ is the smallest possible, namely $s$. However, in order to include the case where $X$ and $Y$ are reducible with $X$ possibly containing some component of $Y$, the following algebraic definition is more suited:

\begin{dfn}
Let $R$ be a polynomial ring and $I$ an ideal of $R$. Given an ideal $\mathfrak{a} \subsetneq I$ generated by $s$ elements, set $J = \mathfrak{a}:I$. If $\height(J) \geq s\geq\height(I)$ holds, then $J$ is an \emph{$s$-residual intersection} of $I$ (or $R/J$ is an $s$-residual intersection of $R/I$). 

If furthermore $I_{\frakp}=\fraka_{\frakp}$ for all $\frakp \in V(I)$ with $\height{\frakp}\leq s$, then $J$ is called a \textit{geometric} $s$-residual intersection of $I$. 
\end{dfn}

Residual intersections are natural generalizations of linkage: if $I$ is unmixed of height $g$, then the $g$-residual intersections of $I$ are precisely its links and the geometric $g$-residual intersections are precisely its geometric links. Furthermore, we may define the `most general' residual intersections by the following technical requirements. 

\begin{dfn}
 We say that an ideal $I$ in a Noetherian ring \emph{satisfies the condition $G_s$} for some natural number $s$ if $\mu(I_{\frakp})\leq \height (\frakp)$ for all prime ideals $\frakp$ containing $I$ such that $\height (\frakp)\leq s-1$. We say that $I$ \emph{satisfies $G_\infty$} if it satisfies $G_s$ for each $s\in\NN$, and this is equivalent to the condition that $\mu(I_{\frakp})\leq \height(\frakp)$ for each prime ideal $\frakp$ containing $I$.  
\end{dfn}

\begin{dfn}
Let $R$ be a polynomial ring and $I$ be an ideal of height $g>0$ satisfying $G_{s+1}$, where $s\geq g$. Choose any generating set $\underline{f}\colonequals f_1, \ldots, f_n$ of $I$ and let $Y$ be an $s \times n$ matrix of indeterminates. Let $\mathfrak{a}$ be the ideal generated by the entries of the matrix $Y[f_1\dots f_n]^{\tr}$. Then we set \[ 
\RI(s;I)=\RI(s;\underline{f})\colonequals \mathfrak{a}R[Y]:IR[Y]\] and call this ideal a \emph{generic $s$-residual intersection} of $I$. Generic residual intersections are essentially independent of the generating set as they are isomorphic up to adjoining finitely many variables due to \cite[Lemma 2.2]{HunekeUlrichGenericRI}. 
\end{dfn}
We note that the arithmetic rank and the cohomological dimension supported at an ideal are unaffected by polynomial extensions (see Remark \ref{rmk:araretract}), so these numbers are well-defined in view of the above definition.

Generic residual intersections are geometric residual intersections \cite[Theorem 3.3]{HunekeUlrichRI} and under mild assumptions, specialize to any $s$-residual intersection of the given ideal \cite[Theorem 1.5]{HunekeUlrichGenericRI}.

\subsection{Algebras with straightening law}\label{subsec-ASLintro}

\begin{dfn}
    Fix  a ring $A$ and a finite poset $\mathcal{H} \subseteq A$. A \textit{standard monomial} is a product of a weakly increasing sequence of elements of $\mathcal{H}$: 
    \[
    \alpha_1 \cdots \alpha_k \ \text{where} \ \alpha_{1} \leq \alpha_{2} \leq \cdots \leq \alpha_{k}.
    \]
    The empty product $1$ is also a standard monomial.
\end{dfn}

\noindent Assume now that $A$ is an $S$-algebra for a ring $S$ and that the elements of $\mathcal{H}$ generate $A$ as an $S$-algebra. 
\begin{dfn}\label{def:ASL}
    The ring $A=S[\mathcal{H}]$ is an \textit{algebra with a straightening law (ASL)} over $\mathcal{H}$ if it satisfies the following axioms:
\begin{itemize}
    \item (ASL-1) The algebra $A$ is a free $S$-module with basis given by the set of standard monomials.
    \item (ASL-2) If $\alpha, \beta \in \mathcal{H}$ are incomparable, and if 
    \[
    \alpha \cdot \beta = \sum_{\substack{c_i\neq 0 \\ \gamma_{i1}\leq \cdots \leq \gamma_{im_i}}}
    c_i \cdot (\gamma_{i1} \gamma_{i2} \cdots \gamma_{im_i})\
    \]
    is the unique expression of $\alpha \cdot \beta$ as a linear combination of standard monomials in $\mathcal{H}$ 
    then $\gamma_{i1} < \alpha$ and $\gamma_{i1} < \beta$ for all $i$. 
    The relations in (ASL-2) are called the\textit{ straightening relations} of the algebra $A$. \qedhere
\end{itemize}
\end{dfn}

\begin{dfn}\label{def:symASL}
    An ASL is \textit{symmetric} if it is also an ASL with the poset order reversed. In other words, a ring $A=S[\mathcal{H}]$ is a symmetric ASL if it satisfies (ASL-1) and (ASL-2) from above with the additional condition that $\alpha<\gamma_{im_i}$ and $\beta<\gamma_{im_i}$ for all $i$. 
\end{dfn}

\begin{dfn}[{\cite[Page 55]{BrunsVetter}}]
   Given a finite poset $\mathcal{H}$ and $\mu\in \mathcal{H}$, define 
   \[
   \text{rank}(\mu)=\max\{k\,|\,\exists \text{ a chain }\mu=\mu_k>\mu_{k-1}>\cdots>\mu_1 \text{ with }\mu_i\in \mathcal{H}\, \, \forall\,  i\}.
   \]
   We also define $\text{rank}(\mathcal{H})=\max\{\text{rank}(\mu):\mu\in \mathcal{H}\}.$ Set $\rank({\emptyset})=0$.
\end{dfn}

\begin{prp}[{\cite[Lemma 5.9]{BrunsVetter}}]\label{prop:SOPofASL}
    Let $\mathcal{H}$ be a finite poset and $A=S[\mathcal{H}]$ be an ASL. Let $x_i=\sum_{\rank(\mu)=i}\mu$ for all $1\le i\le \rank(\mathcal{H})$. Then
    \begin{flalign*}
    &&\sqrt{ (x_1,\dots,x_{\rank(\mathcal{H})})} = \sqrt{(\mathcal{H})}. &&\qed
\end{flalign*}
\end{prp}

\begin{dfn}
    A non-empty subset $\mathcal{C}$ of a poset $\mathcal{H}$ is called an \emph{ideal} if 
    for all $\gamma\in \mathcal{C}$ and $\delta \in \mathcal{H}$ such that $\delta\leq\gamma$, we have $\delta\in \mathcal{C}$.
\end{dfn}

\begin{prp}[{\cite[Proposition 5.22]{BrunsVetter}}]\label{prop:araofposetideal}
    Let $\mathcal{H}$ be a finite poset and $A=S[\mathcal{H}]$ be a symmetric graded ASL. Then for every ideal $\mathcal{C}$ of the poset $\mathcal{H}$, the arithmetic rank of the ideal $(\mathcal{C})$ of $A$ is the rank of $\mathcal{C}$. \qed
\end{prp}

\subsection{\'Etale cohomology}
Let $\KK$ be an algebraically closed field, $V$ be a quasiprojective variety over $\KK$, and $G$ be an Abelian group of order co-prime to $\chr \KK$ (if $\chr\KK>0$).

We denote by $H^i_{\et}(V,G)$ the $i^{\text{th}}$ \'etale cohomology group with coefficients in the constant sheaf on $V$ given by $G$. The Mayer-Vietoris sequence is a critical tool in our \'etale cohomology computations.  

\begin{prp}[{\cite[III.2.24]{Milne}}]\label{prop:et-mv}
    If $V = V_{1} \cup V_{2}$ with $V_{i}$ open in $V$, then there is a long exact sequence
    \[
        \cdots\to H^{i}_{\et}(V,G) \to
        H^{i}_{\et}(V_{1},G) \oplus H^{i}_{\et}(V_{2},G) \to
        H^{i}_{\et}(V_{1} \cap V_{2},G) \to
        H^{i+1}_{\et}(V,G) \to\cdots 
    \]
\end{prp}
The computation of \'etale cohomology groups is relevant in the context of arithmetic rank due to the following proposition. 

\begin{prp}[{\cite[Proposition 2.9]{LyubeznikAnnals}}]\label{prop:ara-from-et}
    Suppose $V$ is a variety of dimension $d$ over $\KK$, and admits a cover by $k$ affine open sets. For any $\ell$-torsion Abelian group $G$ with $\ell$ invertible in $\KK$, 
    \[
    H^i_{\et}(V,G)=0 \text{ for } i>d+k-1.
    \]
\end{prp}

We also denote by $\Hc^i(V,G)$ the $i^{\text{th}}$ \'etale cohomology group with compact support, and define 
\[
\cptdim(V)\colonequals \inf\{i \ge 0 \mid \Hc^{i}(V,G) \neq 0\}.
\]

\begin{prp}[Poincar\'e duality, {\cite[VI.11.1]{Milne}}]\label{prop:Poincareduality}
    If $V$ is smooth of algebraic dimension $d$ and $G=\ZZ/\ell\ZZ$ with $\ell$ invertible in $\KK$, then there is an isomorphism $H_{\text{\'et}}^{i}(V,G) \cong \Hc^{2 \dim(V) - i}(V,G)$ for every $i$. 
\end{prp}

\begin{prp}\label{prop:les-triple}
    If $V = W \sqcup U$ with $U$ open in $V$, then there is a long exact sequence 
    \begin{equation*}
        \cdots\to\Hc^{i}(U,G) \to \Hc^{i}(V,G) \to \Hc^{i}(W,G) \to \Hc^{i+1}(U,G) \to\cdots
    \end{equation*}
\end{prp}

\subsection{Specialization}

\begin{dfn}
    Let $R= \bigoplus_{i \geq 0} R_i$ be a graded ring, $I\subseteq R$ be a proper homogeneous ideal, and $\underline{x}\subseteq R$ be a homogeneous regular sequence on both $R$ and $R/I$. We say that the ideal $I(R/(\underline{x}))=(I+(\underline{x}))/(\underline{x})$ in the ring $R/(\underline{x})$ is a \textit{homogeneous specialization} of $I$, and we also formulate this as \emph{the ideal $I$ specializes to} $I(R/(\underline{x}))$. If $\underline{x}\subseteq R_1$, we say that $I(R/(\underline{x}))$ is a \textit{linear specialization} of $I$. 
    
    As all specializations constructed in this paper are homogeneous, we often use the term `specialization' instead of `homogeneous specialization'.
\end{dfn}

\begin{prp}\label{prop:specialization}
    Let $\KK$ be a field or the integers, and let $R$ be an equidimensional finitely generated non-negatively graded $\KK$-algebra with $R_0=\KK$. Let $I\subseteq R$ be a proper homogeneous ideal such that $R/I$ is Cohen--Macaulay. Suppose that $\underline{x}\subseteq R$ is a homogeneous regular sequence on $R$ and assume that $\height(I)=\height(I(R/(\underline{x})))$. Then $\underline{x}$ is a regular sequence on $R/I$ and thus $I$ specializes to $I(R/(\underline{x}))$.
\end{prp}
\begin{proof}
    Let $\ell$ be the number of elements in $\underline{x}$. Since $\underline{x}$ is regular on $R$, 
    \[
    \dim(R/(\underline{x}))=\dim R - \ell.
    \]
    Then,
    \begin{align*}
        \dim(R/(I,\underline{x}))
        &=\dim(R/\underline{x})-\height(I(R/(\underline{x})))\\
        &=\dim R - \ell - \height(I)\\
        &=\dim(R/I)-\ell.
    \end{align*}
    As $R/I$ is Cohen--Macaulay, we can conclude that $\underline{x}$ is a regular sequence on $R/I$. 
\end{proof}

\begin{exa}
    Let $R=\KK[a,b,c,d,e,f]$ and $I=I_2\begin{pmatrix}
        a & b & c\\
        d & e & f
    \end{pmatrix}=(ae-bd,af-cd,bf-ce).$ It is well-known that $\height(I)=2$ and that $R/I$ is Cohen--Macaulay. Consider two sequences $\underline{x}:=c,f$ and $\underline{y}:=b-d,c-e$; clearly both are regular sequences on $R$. However, $\underline{x}$ is not a regular sequence on $R/I$ as $bf=0$ in $R/(I+(c))$. Moreover, 
    \[
        \height(I(R/(\underline{x})))=\height(ae-bd)=1<\height(I).   
    \]
    On the other hand, 
    \[
        \height(I(R/(\underline{y})))=\height(ac-b^2,af-bc,bf-c^2)=2,
    \] 
    since $R/(I+(\underline{y}))\cong \KK[u^3,u^2v,uv^2,v^3]$ is a Veronese subring of a polynomial ring in two variables.  Proposition \ref{prop:specialization} then confirms that the Veronese ideal $(ac-b^2,af-bc,bf-c^2)$ in $\KK[a,b,c,f]$ is a linear specialization of $I$ (without testing that $b-d,c-e$ is regular on $R/I$ explicitly).
\end{exa}

\section{Arithmetic rank of generic residual intersections of submaximal Pfaffians}
\label{sec:genericRI} 

\begin{setup}\label{setup}
Choose integers $n,s$ with $n\geq 2$ and $s\geq3$. Let  $X_n=(x_{ij})_{i<j}$ be an $n \times n$ alternating matrix of indeterminates and  $Y_{n,s}$  an $n \times s$ matrix of indeterminates. Let $\KK$ be a field or the integers, and denote 
\[
R_{n,s}\coloneqq\KK[X_n, Y_{n,s}].
\]
Then consider the matrix 
\[
M_{n,s}=\begin{pmatrix}X_{n}&Y_{n,s}\\-Y_{n,s}^{\tr}&0_{s,s}\end{pmatrix}
\] 
and define $P(n,s)\subseteq R_{n,s}$ to be the ideal generated by the Pfaffians of $M_{n,s}$ (of all possible sizes) that involve at least the entire top left block $X_{n}$.  Moreover, let $I(n,s)$ denote the ideal generated by only those Pfaffian generators of $P(n,s)$ that are not $n$-minors of $Y$.  Set \[
d(n,s)=\binom{s}{2}-\binom{s-n}{2}+1.
\]
When there is no danger of confusion, we may drop the labels $n$ and $s$ from the above notations.
\end{setup}

\begin{exa}\label{exa:P(n,s)}
    For $n=2$ and $s=4$, 
    \[M_{2,4}=\begin{pmatrix}
              0 & x_{12}  & y_{11} & y_{12} & y_{13} & y_{14}\\
        -x_{12} & 0       & y_{21} & y_{22} & y_{23} & y_{24}\\
        -y_{11} & -y_{21} & 0      & 0      & 0      & 0     \\
        -y_{12} & -y_{22} & 0      & 0      & 0      & 0     \\
        -y_{13} & -y_{23} & 0      & 0      & 0      & 0     \\
        -y_{14} & -y_{24} & 0      & 0      & 0      & 0     
    \end{pmatrix}.\]
    Note that $M_{2,4}$ has one Pfaffian of size two containing $X$, $\binom{4}{2}=6$ such of size four, and one such of size six. The size two Pfaffian is $x_{12}$, the size four Pfaffians are the various $y_{1i}y_{2j}-y_{1j}y_{2i}$, and the size six one is $0$. Thus 
    \[
    P(2,4)=(x_{12},\{y_{1i}y_{2j}-y_{1j}y_{2i}\:|\: 1\leq i<j\leq 4\}).
    \]

    For $n=5$ and $s=8$, $M$ is a $13\times 13$ alternating matrix. Here, $P(5,8)$ has as its generators $\binom{8}{1}=8$ Pfaffians of size six, $\binom{8}{3}=56$ Pfaffians of size eight, and $\binom{8}{5}=56$ Pfaffians of size ten. The alternating submatrices of $M_{5,8}$ of size twelve containing $X$ have `too many' zeroes (more specifically, the entire bottom right $7\times 7$ submatrix has only zero entries) and hence, Pfaffians equal to zero.
\end{exa}

We now make some remarks to begin our cohomological dimension and arithmetic rank computations.

\begin{rmk}\label{rmk:n=2}
For $n=2$, observe that
\[P(2,s)=(x_{12})+I_2 
\begin{pmatrix}
    y_{11} & y_{12} & \cdots & y_{1s} \\
    y_{21} & y_{22} & \cdots & y_{2s}
\end{pmatrix}.
\] The height of the ideal of $2$-minors of a $2\times s$ matrix is $s-1$ and the arithmetic rank is $2s-3$ \cite{BrunsSchwanzl}. Moreover, the arithmetic rank is detected by the nonvanishing of local cohomology (in characteristic zero) or \'etale cohomology (in any characteristic). In particular, $\height(P(2,s))=s$ and $\ara(P(2,s))=2s-2$ (see Remark \ref{rmk:ara-realized-by-local-cohomology}).  
\end{rmk}

\begin{rmk}\label{rmk:P(3,s)}
For $n=3$, a close look at the generators yields that
\[
P(3,s)=I_1\bigg(Y^{\tr}\begin{pmatrix}
    x_{23}\\
    -x_{13}\\
    x_{12}
\end{pmatrix}\bigg)+I_3(Y),
\] which equals the generic residual intersection of the ideal of variables $(x_{13},x_{23},x_{12})$ (cf. \cite[Example 3.4]{HunekeUlrichRI}). The arithmetic rank of generic residual intersections of ideals of variables has been computed \cite[Theorem 4.7]{BMSMP}, and in this case, equals $3s-5$. 
\end{rmk}

\begin{rmk}
Observe that the $n$-minors of $Y$, if they exist, are among the generators of $P(n,s)$. Indeed, if we consider the alternating submatrix of $M_{n,s}$ corresponding to the columns of $X$ and any $n$ columns of $Y$, the Pfaffian of this submatrix is the $n$-minor of $Y$ defined by the chosen $n$ columns of $Y$. However, if we were to consider the submatrix corresponding to $X$ and $m$ columns of $Y$, where $m>n$, the resulting Pfaffian would be $0$, as in Example \ref{exa:P(n,s)}.
\end{rmk}

In Setup \ref{setup}, we defined $I(n,s)$ to be the ideal generated by all the generators of $P(n,s)$ except the $n$-minors of $Y$. From the above discussion, we have that 
\[
P(n,s)=I(n,s)+I_n(Y) \text{ and } I(n,s)\subseteq(X).\qedhere
\] 
Denote each generator of $P(n,s)$ by indicating the columns of $Y$ chosen in the corresponding alternating submatrix. For instance, $[\;]$ denotes the Pfaffian of the submatrix $X$, and $[1,4,5]$ denotes the Pfaffian of the alternating submatrix formed by $X$ and the first, fourth and fifth columns of $Y$. Observe that if $n$ is odd, then $[i_1,\ldots,i_t]=0$ if $t$ is even and if $n$ is even, $[i_1,\ldots,i_t]=0$ if $t$ is odd. Denote the set of nonzero generators of $P(n,s)$ by $\mathcal{P}$, that is,
\[ 
\mathcal{P}\colonequals \{[i_1,\ldots,i_t]\:|\:0\leq t\leq n,\: t\equiv n \text{ mod } 2,\: 1\leq i_1<\cdots<i_t\leq s\}. 
\]

Similarly, denote the set of nonzero generators of $I(n,s)$ as $\mathcal{I}$. 

We next produce an upper bound on the arithmetic rank of $P(n,s)$ by producing a set of set-theoretic generators of $P(n,s)$. To do so, we analyse the structure of the `auxiliary ring' $\KK[\mathcal{P}]$. Assign 
\[
\deg(x_{i,j})=2 \text{ and }\deg(y_{i,j})=1;
\]
then each element of $\mathcal{P}$ (and $\mathcal I$) is homogeneous and $\KK[\mathcal{P}]$ is a non-negatively graded $\KK$-algebra. 

\subsection{Upper bound: The auxiliary ring $\KK[\mathcal{P}]$ is a symmetric ASL}

We refer to Subsection \ref{subsec-ASLintro} for language, and define a partial order $<$ on $\mathcal{P}$ as follows:
\[
[i_1,\ldots,i_t]\leq[j_1,\ldots,j_s] \text{ iff } t\geq s \text{ and } i_k\leq j_k \text{ for all } k\in\{1,\ldots,s\}.
\]

\begin{exa}
    Let $n=4$ and $s=4.$ The partial order on $\mathcal{P}$ is depicted by the following Hasse diagram:
    \[
    \begin{tikzcd}[cramped, column sep=small, row sep=small,baseline=(bottom.base)]
        && {[\;]} && \\
        \\
        && {[3,4]} \\
        \\
        && {[2,4]} \\
        \\
        {[1,4]} &&&& {[2,3]} \\
        \\
        && {[1,3]} \\
        \\
        && {[1,2]} \\
        \\
        && |[alias=bottom]| {[1,2,3,4]}
        \arrow[no head, from=1-3, to=3-3]
        \arrow[no head, from=3-3, to=5-3]
        \arrow[no head, from=5-3, to=7-1]
        \arrow[no head, from=5-3, to=7-5]
        \arrow[no head, from=7-1, to=9-3]
        \arrow[no head, from=7-5, to=9-3]
        \arrow[no head, from=9-3, to=11-3]
        \arrow[no head, from=11-3, to=13-3]
    \end{tikzcd} \qedhere
    \]
\end{exa}

\noindent We now introduce a result on Pfaffians which follows directly from an identity proved by Kustin. 

\begin{lem}\label{lem:pfaffstr}
    Let $M$ be a generic alternating matrix with entries from a Noetherian ring $R$. Define a partial order $<$ on the set of Pfaffians of $M$ (as above): 
    \[
    [i_1,\ldots,i_t]\leq[j_1,\ldots,j_s] \text{ iff } t\geq s \text{ and } i_k\leq j_k \text{ for all } k\in\{1,\ldots,s\}.
    \]
    Given any two incomparable Pfaffians $\alpha$ and $\beta$ of $M$, we can express $$\alpha\beta=\sum_{i}c_i\gamma_i\delta_i,$$ where $c_i\in \ZZ$, and $\gamma_i$ and $\delta_i$ are Pfaffians of $M$ such that $\alpha>\gamma_i$ and $\beta<\delta_i$ for all $i$. Here we use the convention that $[\emptyset]=1\in R$, where $\emptyset$ is the empty set.

    Moreover, the sum can be arranged such that whenever an index $j$ occurs in both $\alpha$ and $\beta$, it occurs in both $\gamma_i$ and $\delta_i$ for every nonzero summand.
\end{lem}
\begin{proof}
    Let $\alpha=[\alpha_1,\ldots,\alpha_q]$ and $\beta=[\beta_1,\ldots,\beta_r]$. As Pfaffians of odd order are zero, we may assume that $q$ and $r$ are even. If $q<r$ and $\alpha_i\leq\beta_i$ for $i\in\{1,\ldots,q\}$, let $A=\{\alpha_1,\ldots,\alpha_q\}$, $B=\{\beta_{q+2},\ldots,\beta_r\}$, and $C=\{\beta_1,\ldots,\beta_{q+1}\}.$ Applying \cite[Proposition 6.6]{Kustin} with $M=0$ and $Q=r-q-2$ gives us the desired result.
    
    In all the remaining cases, there exists $0\leq i<\min\{q,r\}$ such that $\alpha_j\leq\beta_j$ for all $j\leq i$ and $\alpha_{i+1}>\beta_{i+1}$. Here, let $A=\{\alpha_1,\ldots,\alpha_{i}\}$, $B=\{\beta_{i+2},\ldots,\beta_{r}\},$ and $C=\{\alpha_{i+1},\ldots,\alpha_q,\beta_1,\ldots,\beta_{i+1}\}$. Once again, an application of \cite[Proposition 6.6]{Kustin} with $M=q-i$ and $Q=r-i-2$ produces our result. 
\end{proof}

\begin{rmk}
    Lemma \ref{lem:pfaffstr} provides a shuffling relation on
    Pfaffians that addresses a gap in the derivation of the
    classical Pfaffian straightening law from the shuffling
    relations in \cite[Equation (6.3)]{Kustin} or
    \cite[Theorem]{ChiriviMaffei}.

    These relations are presented with the longer factor first, and their terms need not satisfy the order inequalities required of a straightening relation. To obtain both inequalities $\gamma_{i1}<\alpha$ and $\gamma_{i1}<\beta$ (see Definition \ref{def:ASL}), one straightens $\alpha\beta$ and $\beta\alpha$ separately and compares the resulting
    standard expansions using the linear independence of standard monomials (see for instance, the proof of Theorem \ref{thm:ASL}). This argument requires shuffling relations with the roles of the factors fixed, including when the first factor is shorter than the second. Lemma \ref{lem:pfaffstr} supplies the missing orientation and thereby completes this argument.

    The classical Pfaffian straightening law establishes that the $\KK$-algebra generated by the Pfaffians of a generic alternating matrix is an ASL. This result has been used extensively in the literature; see, for example, \cite{Barile,DeNegriSbarra}.
\end{rmk}

We represent each standard monomial in $\mathcal{P}$ as a tableau. Write the factors of a given standard monomial in weakly increasing order in $\mathcal{P}$. The Pfaffian $[i_1,\ldots,i_t]$ contributes 
a row with entries $i_1,\ldots,i_t$ to the tableau. If the element $[\;]$ is a factor, it contributes a row with the single entry $s+1$. For instance, with $s=n=4$ as in the example above, the product $[1,2,3,4][1,4]^2[2,4][3,4][\;]$ is represented by the diagram 
\[
\begin{ytableau} 
    1 & 2 & 3 & 4 \\ 
    1 & 4 \\ 
    1 & 4 \\
    2 & 4 \\ 
    3 & 4 \\
    5
\end{ytableau}
\]
The definition of the order shows that in a standard tableau: the row lengths weakly decrease, the entries in each row are strictly increasing, and the entries in each column are weakly increasing.

\begin{thm}\label{thm:ASL}
    The auxiliary ring $\KK[\mathcal{P}]$ is a symmetric ASL (see Definition \ref{def:symASL}). 
\end{thm}
\begin{proof} 
    We begin by showing that the standard monomials in $\mathcal{P}$ are linearly independent over $\KK$. To do so, we consider the lexicographic monomial order on $\KK[X,Y]$ defined by 
    $$y_{1,1}>y_{1,2}>\cdots>y_{1,s}>y_{2,1}>\cdots>y_{n,s}>x_{1,2}>\cdots>x_{1,n}>x_{2,3}>\cdots>x_{n-1,n}.$$ The initial monomial of $[i_1,\ldots,i_t]$ in $\KK[X,Y]$ with respect to this monomial order is $$\init([i_1,\ldots,i_t])=y_{1,i_1}y_{2,i_2}\ldots y_{t,i_t}x_{t+1,t+2}x_{t+3,t+4}\ldots x_{n-1,n}.$$
    To prove the $\KK$-linear independence of the standard monomials, it suffices to prove that every standard monomial has a distinct initial monomial. Consider a standard monomial $\mu$ represented by a standard tableau $T$. Define 
    \[
    d^X_{i,j}\colonequals\deg_{x_{i,j}}(\init(\mu)) \text{ and } d^Y_{i,j}\colonequals\deg_{y_{i,j}}(\init(\mu))
    \]
    The following observations show that the initial monomial of $\mu$ completely determines the multiset of entries in the tableau $T$:
    \begin{enumerate}
        \item The number of rows of $T$ is equal to the sum $$d^X_{1,2}+\sum_{i=1}^{s}d^Y_{1,i}.$$ The term $d^X_{1,2}$ is nonzero exactly if $[\;]$ is one of the elements of $\mathcal{P}$ appearing in the product $\mu$.
        \item If $d^X_{1,2}\neq0$, the multiset of entries in the first column of $T$ is $$\{s+1\}^{d^X_{1,2}}\cup\bigcup_{i=1}^s \{i\}^{d^Y_{1,i}},$$ where $\{i\}^d$ is the multiset with $d$ entries of $i$. If $d^X_{1,2}=0$, the multiset of entries in the first column of $T$ is $\bigcup_{i=1}^s \{i\}^{d^Y_{1,i}}$.
        \item For $j>1$, the multiset of entries in the $j^{th}$ column of $T$ is $$\bigcup_{i=1}^s \{i\}^{d^Y_{j,i}}.$$ 
    \end{enumerate}

    As any multiset can be uniquely arranged as a weakly increasing sequence and the columns of standard tableaux must be weakly increasing, the initial monomial $\init(\mu)$ uniquely defines a standard tableau $T$, and vice-versa. Therefore distinct standard monomials have distinct initial monomials and so the initial monomial of $\mu$ determines the tableau. 
    
    The fact that the standard monomials span $\KK[\mathcal{P}]$ as a $\KK$-module follows from the argument below which also proves that the products of incomparable elements in $\mathcal{P}$ can be `straightened'. Let $\alpha$ and $\beta$ be incomparable elements of $\mathcal{P}$. By Lemma \ref{lem:pfaffstr}, $$\alpha\beta=\sum_{i}c_i\gamma_i\delta_i$$ for $c_i\in \ZZ$, $\gamma_i,\delta_i\in \mathcal{P}$ such that $\alpha>\gamma_i$ and $\beta<\delta_i$ for all $i$. If $\gamma_i<\delta_i$ for all $i$, we are done. Suppose $\gamma_i\not<\delta_i$ for some $i$. As $\alpha$ and $\beta$ are incomparable, $\gamma_i\not>\delta_i$. Hence, $\gamma_i$ and $\delta_i$ are incomparable, and we may once again apply Lemma \ref{lem:pfaffstr} to the product $\gamma_i\delta_i$. 
    
    As $\mathcal{P}$ is a finite poset, we eventually obtain that $\alpha\beta$ is an integer combination of standard monomials $\gamma\delta$ such that $\gamma<\delta$, $\alpha>\gamma$ and $\beta<\delta$. We can now do the same for the product $\beta\alpha$. As the standard monomials are $\KK$-linearly independent (and $\alpha\beta=\beta\alpha$ in the commutative ring $\KK[\mathcal{P}])$, we obtain the same linear combination. This shows that for every standard monomial $\gamma\delta$ appearing in the expression of $\alpha\beta$, we have $\alpha,\beta>\gamma$ and $\alpha,\beta<\delta.$
\end{proof}

We can now leverage Proposition \ref{prop:SOPofASL} to obtain the dimension of $\KK[\mathcal{P}]$. 

\begin{cor}\label{cor:dimauxring}
     The ideal $(\mathcal{P})$ in $A=\KK[\mathcal{P}]$ is generated up to radical by the set
    \[
    \left\{ \sum_{\substack{\operatorname{rank}(\mu)=i}} \mu \mid i=1,\ldots,\binom{s}{2}-\binom{s-n}{2}+1 \right\}.
    \]
    Thus, the Krull dimension of $\KK[\mathcal{P}]$ is $\binom{s}{2}-\binom{s-n}{2}+1$ if $\KK$ is a field, and $\binom{s}{2}-\binom{s-n}{2}+2$ if $\KK=\ZZ$.
    
\end{cor}
\begin{proof}
    By Proposition \ref{prop:SOPofASL}, it suffices to show that the poset $\mathcal P$ is ranked and has rank $\binom{s}{2}-\binom{s-n}{2}+1.$ 
    
    The function 
    \[
    w([i_1,\ldots,i_t])=\sum_{k=1}^t (i_k-s)
    \]
    (with $w([\;])=0$) increases by exactly one along any upward covering relation. Thus, $\mathcal{P}$ is ranked, and the length of the longest chain in $\mathcal{P}$ (its rank), is one more than the difference between the value of $w$ at the smallest and largest elements of $\mathcal{P}$. The smallest element of $\mathcal P$ is $[1,\ldots,q]$; the value of $q$ depends on the relative sizes and parities of $s$ and $n$, which we address later. The largest element of $\mathcal P$ is $[s]$ if $n$ is odd, and $[\;]$ otherwise. Hence, 
    \[
        \rank(\mathcal P)=qs-\binom{q+1}{2}+1.
    \]

    If $s\geq n$, then $q=n$. If $s<n$ and $s\equiv n \text{ mod }2$, then $q=s$, otherwise $q=s-1$. In all cases, the rank evaluates to the desired value.
\end{proof}

Recall that $d(n,s)$ denotes the quantity $\binom{s}{2}-\binom{s-n}{2}+1$.

\begin{cor}\label{cor:araI(n,s)}
    The arithmetic rank of the ideal $(\mathcal{I})$ in $\KK[\mathcal{P}]$ is $$\binom{s}{2}-\binom{s-n+2}{2}+1=d(n-2,s).$$ 
\end{cor}
\begin{proof}
    Recall that 
    \[
    \mathcal{I}=\{[i_1,\ldots,i_t]\:|\:0\leq t<n,\: t\equiv n \text{ mod } 2,\: 1\leq i_1<\cdots<i_t\leq s\}.
    \] 
    Then $\mathcal I$ is an ideal of the poset $\mathcal P$ under the $>$ that is reverse to the poset order $<$ on $\mathcal{P}$. As $\mathcal{P}$ is a symmetric ASL, 
    \begin{align*}
    \ara((\mathcal{I}))=\operatorname{rank }\mathcal{I} &=\binom{s}{2}-\binom{s-n+2}{2}+1\\
    &=d(n-2,s)
    \end{align*} by Proposition \ref{prop:araofposetideal}. Here, the rank computation for $\mathcal{I}$ follows directly from Corollary \ref{cor:dimauxring} for $s<n$ and can be concluded by a similar maximal chain construction for $s\geq n$. 
\end{proof}

\noindent By Corollary \ref{cor:dimauxring} and Corollary \ref{cor:araI(n,s)}, we can conclude that $\ara(P(n,s))\leq d(n,s)$ and $\ara(I(n,s))\leq d(n-2,s).$

\subsection{Lower bound on arithmetic rank: Cohomological obstructions}

The following lemma is a key component in our arithmetic rank computation as well as in the discussion on the top local cohomology module in Section \ref{sec:cohomology}.

\begin{lem}\label{lem:localization}
For $n>2$, let $X_{n}$, $Y_{n,s}$ and $R\coloneqq R_{n,s}$ be as in Setup \ref{setup}. Let $X'_{n-1}=(x'_{i,j})_{i<j}$ be a generic alternating $(n-1)\times(n-1)$ matrix, $Y'_{n-1,s}=(y'_{i,j})$ be a generic $(n-1)\times s$ matrix, and $R'\coloneqq \KK[X',Y']$. We define the ideal $P'(n-1,s)$ in $R'$ as in Setup \ref{setup}. Then the ring map 

\begin{eqnarray*}
    \Phi\colon R'[y_{n,1},\dots,y_{n,s-1},y_{1,s},\dots,y_{n-1,s},y_{n,s}]_{y_{n,s}}&\lto&R_{y_{n,s}},\\
    \text{with}\phantom{\rule{5cm}{0cm}} x'_{i,j}&\mapsto& x_{i,j}+\frac{y_{i,s}x_{j,n}-y_{j,s}x_{i,n}}{y_{n,s}},\\
    y'_{i,j}&\mapsto&y_{i,j}-\frac{y_{i,s}y_{n,j}}{y_{n,s}}\quad \text{ for } j<s,\\
    y_{i,j}&\mapsto&y_{i,j}\quad \text{ for } j=s\text{ or }i=n,\\
    y'_{i,s}&\mapsto&x_{i,n}, 
    \end{eqnarray*}
    is an isomorphism. Moreover, $R_{y_{n,s}}$ is a free $R'$-module  and $$P(n,s)R_{y_{n,s}}=P'(n-1,s)R_{y_{n,s}}$$ under the above isomorphism.
\end{lem}
\begin{proof}
    Upon inverting $y_{n,s}$, we may perform (simultaneous) row and column operations to transform $M_{n,s}$ into a matrix of the form 
    \[
    M'=\begin{pmatrix}
    0      & x'_{12} & \cdots & x'_{1,n-1} & x_{1n} & y'_{11} & \cdots & y'_{1,s-1} & 0 \\
    -x'_{12}& 0      & \cdots & x'_{2,n-1} & x_{2n} & y'_{21} & \cdots & y'_{2,s-1} & 0 \\
    \vdots & \vdots & \ddots & \vdots     & \vdots & \vdots & \ddots & \vdots     & \vdots \\
    -x'_{1,n-1} & -x'_{2,n-1} & \cdots & 0 & x_{n-1,n} & y'_{n-1,1} & \cdots & y'_{n-1,s-1} & 0 \\
    -x_{1n} & -x_{2n} & \cdots & -x_{n-1,n} & 0 & 0 & \cdots & 0 & y_{n,s} \\

    - y'_{11} & -y'_{21} & \cdots & -y'_{n-1,1} & 0 & 0 & \cdots & 0 & 0 \\
    \vdots & \vdots & \ddots & \vdots & \vdots & \vdots & \ddots & \vdots & \vdots \\
    - y'_{1,s-1} & -y'_{2,s-1} & \cdots & -y'_{n-1,s-1} & 0 & 0 & \cdots & 0 & 0 \\
    0 & 0 & \cdots & 0 & -y_{n,s} & 0 & \cdots & 0 & 0
    \end{pmatrix}
    \] where $x'_{ij}=x_{ij}+\frac{y_{i,s}x_{j,n}-y_{j,s}x_{i,n}}{y_{n,s}}$ and $y'_{ij}= y_{ij}-\frac{y_{i,s}y_{n,j}}{y_{n,s}}$. Note that the operations performed are such that the Pfaffians of the transformed matrix $M'$ containing the top left $n\times n$ block still generate $P(n,s)R_{y_{n,s}}.$ As $y_{n,s}$ is a unit in $R_{y_{n,s}}$, we may rescale the last row and column by $y_{n,s}^{-1}$ in $M'$ without changing the ideal under consideration.  Consider the matrix $M''$ below obtained by deleting the last row and column of $M'$, and moving the $n^{\text{th}}$ row and column to the bottom and right ends respectively.
    \[
    M''=\begin{pmatrix}
    0      & x'_{12} & \cdots & x'_{1,n-1}
           & y'_{11} & \cdots & y'_{1,s-1} & x_{1n} \\
    -x'_{12} & 0     & \cdots & x'_{2,n-1}
           & y'_{21} & \cdots & y'_{2,s-1} & x_{2n} \\
    \vdots & \vdots & \ddots & \vdots
           & \vdots & \ddots & \vdots & \vdots \\
    -x'_{1,n-1} & -x'_{2,n-1} & \cdots & 0
           & y'_{n-1,1} & \cdots & y'_{n-1,s-1} & x_{n-1,n} \\
    
    -y'_{11} & -y'_{21} & \cdots & -y'_{n-1,1}
           & 0 & \cdots & 0 & 0 \\
    \vdots & \vdots & \ddots & \vdots
           & \vdots & \ddots & \vdots & \vdots \\
    -y'_{1,s-1} & -y'_{2,s-1} & \cdots & -y'_{n-1,s-1}
           & 0 & \cdots & 0 & 0 \\
    
    -x_{1n} & -x_{2n} & \cdots & -x_{n-1,n}
           & 0 & \cdots & 0 & 0
    \end{pmatrix}
    \]
    Observe that the set of Pfaffians of $M'$ containing the top left $n\times n$ block is equal (up to signs) to the set of Pfaffians of $M''$ containing the top left $(n-1)\times (n-1)$ block. The ideal generated by these Pfaffians of $M''$ is precisely $P'(n-1,s)$. The map $\Phi$ expresses the above row and column operations in an exact manner.
\end{proof}

\begin{cor}\label{cor:P(n,s)primeCM}
    For all $s$ and $n$, $P(n,s)$ is a prime ideal of height $s$ and it defines a Cohen--Macaulay ring. 
\end{cor}
\begin{proof}
For odd $n$, $P(n,s)$ defines a Cohen--Macaulay ring by \cite[Theorem 1.4(b)]{HunekeUlrichGenericRI}. For even $n$, this follows from Lemma \ref{lem:localization} since any localization of a Cohen--Macaulay ring is Cohen--Macaulay.

Recall that for odd $n$, $P(n,s)$ is a prime ideal of height $s$ \cite[Theorem 3.3]{HunekeUlrichRI}. By Lemma \ref{lem:localization}, $P(n,s)$ is prime for even $n$ as well and 
\[
\height(P(n-1,s))\leq\height(P(n,s))\leq\height(P(n+1,s))
\]
for $n\geq 4$, so that $\height(P(n,s))=s$. The $n=2$ case was discussed in Remark \ref{rmk:n=2}.
\end{proof}

The following result is crucial to the upcoming cohomological calculations.
\begin{lem}\label{lem:PintX}
    If $s\geq n>2$, the minimal primes of $I(n,s)$ are $P(n,s)$ and $(X)$, and thus
    \[
    P(n,s)\cap(X)=\sqrt{I(n,s)}.
    \]
    If $s<n$, then $I(n,s)=P(n,s)$. If $n=2$, then $I(2,s)=(X)$.
\end{lem}
\begin{proof}
    The cases $s<n$ and $n=2$ are clear by definition.  For $s\geq n>2$, as $P(n,s)=I(n,s)+I_n(Y)$, it follows that 
    \begin{align*}                     
        P(n,s)\cap(X)=\sqrt{P(n,s)\cap(X)}&=\sqrt{(I(n,s)+I_n(Y))\cap(X)}\\
                &=\sqrt{(I(n,s)\cap(X))+(I_n(Y)\cap(X))}\\
                &=\sqrt{I(n,s)+I_n(Y)\cdot (X)}.
        \end{align*}
    We show that $I_n(Y)\cdot (X)\subseteq I(n,s).$ As each generator of $I_n(Y)\cdot(X)$ is of the form $x_{p,q}[i_1,\ldots,i_n]$, it suffices to prove that $x_{p,q}[i_1,\ldots,i_n]\in I(n,s).$ 
    
    Let $\sign(p,q,\alpha,\beta)$ be the sign of the permutation $\sigma\in S_n$, where $\sigma(p)=\alpha$, $\sigma(q)=\beta$, and $\sigma(i)$ for $i\neq p,q$ is defined inductively as the smallest available number in $[n]\smallsetminus\{\alpha,\beta\}$. We use $[k_1,\ldots,k_t\:|\:l_1,\ldots,l_t]_Y$ to denote the minor of $Y$ defined by choosing the rows and columns with indices $k_1,\ldots,k_t$ and $l_1,\ldots,l_t$ respectively. Observe that 
    \begin{equation}\label{eqn:(n-2)Pfaff}
    \begin{split}
        [j_1,\ldots,j_{n-2}]&=x_{1,2}[3,\ldots,n\:|\:j_1,\ldots,j_{n-2}]_Y-x_{1,3}[2,4,\ldots,n\:|\:j_1,\ldots,j_{n-2}]_Y+\ldots\\
        &=\sum_{1\leq u<v\leq n}(-1)^{u-v+1} x_{u,v}[1,\ldots,\hat{u},\ldots,\hat{v},\ldots,n\:|\:j_1,\ldots,j_{n-2}]_Y.\\
    \end{split}
    \end{equation}
    Now consider the sum $$\sum_{1\leq\alpha<\beta\leq n}\sign(p,q,\alpha,\beta)[p,q\:|\:i_\alpha,i_\beta]_Y[i_1,\ldots,\hat{i_\alpha},\ldots,\hat{i_\beta},\ldots,i_n].$$
    By \ref{eqn:(n-2)Pfaff}, the coefficient of $x_{p,q}$, up to a change of sign, in this sum is $$\sum_{1\leq\alpha<\beta\leq n}\sign(p,q,\alpha,\beta)[p,q\:|\:i_\alpha,i_\beta]_Y[1,\ldots,\hat{p},\ldots,\hat{q},\ldots,n\:|\:i_1,\ldots,\hat{i_\alpha},\ldots,\hat{i_\beta},\ldots,i_n]_Y,$$ which is precisely the Laplace expansion of the determinant $[i_1,\dots,i_n]$ along the rows $p$ and $q$. Similarly, the coefficient of $x_{u,v}$ for $(u,v)\neq (p,q)$ is the Laplace expansion of the determinant of a matrix in which the $p^{th}$ row or the $q^{th}$ row is repeated, and hence zero. We thus conclude that 
    \[
    x_{p,q}[i_1,\ldots,i_n]=\sum_{1\leq\alpha<\beta\leq n}\sign(p,q,\alpha,\beta)[p,q\:|\:i_\alpha,i_\beta]_Y[i_1,\ldots,\hat{i_\alpha},\ldots,\hat{i_\beta},\ldots,i_n]
    \]
    and hence, $x_{pq}[i_1,\ldots,i_n]\in I(n,s).$
\end{proof}

\begin{thm}\label{thm:cdP(n,s)}
    Assume that $\KK$ has characteristic $0$. For all $s\geq 3$ and $n\geq 2$, 
    \[
    \lcd(P(n,s))=\ara(P(n,s))=\binom{s}{2}-\binom{s-n}{2}+1.
    \] 
\end{thm}
\begin{proof}
    As the cases $n=2$ and $n=3$ have already been discussed in Remarks \ref{rmk:n=2} and \ref{rmk:P(3,s)} respectively, let $n>3$. As before, we utilize the notation $d(n,s)=\binom{s}{2}-\binom{s-n}{2}+1.$ By Corollary \ref{cor:dimauxring}, we have $\lcd(P(n,s))\leq\ara(P(n,s))\leq d(n,s).$ We split our analysis into two cases: $s>n$ and $s\leq n$.

    \noindent \textbf{Case 1:} $s>n$. 

    \noindent For notational simplicity, let $R=R_{n,s}$, $P=P(n,s)$, $I=I(n,s)$, and $d=d(n,s)$. Consider the Mayer-Vietoris sequence on local cohomology with the ideals $P$ and $(X).$ We have $P+(X)=I+I_n(Y)+(X)=I_n(Y)+(X)$ and, by Lemma \ref{lem:PintX}, $P\cap(X)=\sqrt{I}$. Note that $\lcd(I)\leq\ara(I)\leq d(n-2,s)<d-1$ for $s>n$ by Corollary \ref{cor:araI(n,s)}. Also, $\lcd(X)=\binom{n}{2}<d$ and $\lcd(I_n(Y)+(X))=sn-n^2+1+\binom{n}{2}=d$. Then, the following snippet of the Mayer-Vietoris sequence of local cohomology 
    \[H^{d-1}_I(R)\to H^d_{I_n(Y)+(X)}(R)\to H^d_P(R)\oplus H^d_{(X)}(R)\to H^d_I(R)\] implies that $H^d_P(R)\cong H^d_{I_n(Y)+(X)}(R) \neq 0.$ Hence, $\lcd(P)\geq d$ and we conclude that $\lcd(P)=\ara(P)=d.$

    \noindent \textbf{Case 2:} $s\leq n$.
    
    \noindent By Lemma \ref{lem:localization}, $$\lcd(P(n,s))\geq\lcd(P(n-1,s))\geq\ldots\geq\lcd(P(s-1,s))=\binom{s}{2}+1.$$ Observe that $d(n,s)=\binom{s}{2}+1$ if $s\leq n.$ So, $\lcd(P(n,s))=\ara(P(n,s))=d(n,s).$ 
\end{proof}

\begin{rmk}
  By Corollary \ref{cor:P(n,s)primeCM}, $P(n,s)$ defines a Cohen--Macaulay ring and hence, if $\chr\KK=p>0$, we have $$\lcd(P(n,s))=\height(P(n,s))=s,$$ by \cite[Proposition~III.4.1]{PeskineSzpiro}. 
\end{rmk}

\begin{thm}\label{thm:ara-in-char-p}
    Assume that $\KK$ has positive characteristic $p$. For all $s\geq 3$ and $n\geq 2$, \[\ara(P(n,s))=\binom{s}{2}-\binom{s-n}{2}+1.\] 
\end{thm}
\begin{proof}
    As in the proof of Theorem \ref{thm:cdP(n,s)}, it suffices to prove that $\ara(P(n,s))\geq d(n,s)$ for $n>3.$ Note that a lower bound for the arithmetic rank in $\overline{\KK}[X,Y]$ is also a lower bound in $\KK[X,Y]$, where $\overline{\KK}$ is the algebraic closure of $\KK$. Thus we may assume that $\KK$ is algebraically closed. We consider the familiar two cases: $s>n$ and $s\leq n$.

    \noindent \textbf{Case 1:} $s>n$.
    Let $R$, $P$, $I$ and $d$ be as in the proof of Theorem \ref{thm:cdP(n,s)}. Let $q$ be a prime integer distinct from $p$, and $N\colonequals\dim(R)=\binom{n}{2}+sn.$ Define 
    \[
    U:=\KK^N\setminus V(P) \text{ and } V:=\KK^N\setminus V(X).
    \]
    Observe that $H^i_{\et}(V,\ZZ_q)=0$ for all $i>N+\binom{n}{2}-1$ by Proposition \ref{prop:ara-from-et}. By Lemma \ref{lem:PintX}, $U\cap V=\KK^N\setminus V(I)$ and $U\cup V=\KK^N\setminus V(I_n(Y),X)$. Note that $H^{N+d-1}_{\et}(U\cup V,\ZZ_q)\neq 0$ and $H^i_{\et}(U\cup V,\ZZ_q)=0$ for all $i>N+d-1$ (cf. \cite{BrunsSchwanzl} and Proposition \ref{prop:ara-realized-by-etale-cohomology}). Also, as $\ara I\leq d(n-2,s)<d-1$ for $s>n$, we can conclude that $H^{i}_{\et}(U\cap V,\ZZ_q)=0$ for all $i>N+d-3.$ Then, the Mayer Vietoris sequence of \'etale cohomology (Proposition \ref{prop:et-mv})
    \[\to H^{i-1}(U\cap V,\ZZ_q)\to H^{i}_{\et}(U\cup V,\ZZ_q)\to H^{i}_{\et}(U,\ZZ_q)\oplus H^{i}_{\et}(V,\ZZ_q)\to H^{i}_{\et}(U\cap V,\ZZ_q)\to\]
    for $i=N+d-1$ yields that $H^{N+d-1}_{\et}(U,\ZZ_q)\cong H^{N+d-1}_{\et}(U\cup V,\ZZ_q)\neq0.$ Hence, by Proposition \ref{prop:ara-from-et}, $\ara(P(n,s))\geq d.$

    \noindent \textbf{Case 2:} $s\leq n$.
    By Lemma \ref{lem:localization},
    \[\ara(P(n,s))\geq\ara(P(n-1,s))\geq\ldots\geq\ara(P(s-1,s))=\binom{s}{2}+1=d(n,s),\] for $s\leq n.$ 
\end{proof}

One can obtain the cohomological dimension and arithmetic rank of $I(n,s)$ by specializing some matrix entries of $M_{n,s}$. 

\begin{prp}\label{prop:cdI(n,s)}
    For $n>3$, 
    \[
    \ara(I(n,s))=\ara(P(n-2,s))=\binom{s}{2}-\binom{s-n+2}{2}+1.
    \]
    Moreover, if $\chr \KK=0$, then $\lcd(I(n,s))=\ara(I(n,s)).$
\end{prp}
\begin{proof}
    Set $y_{ij}=0$ for $i>n-2$, $x_{n-1,n}=1$ and $x_{ij}=0$ for all other $i<j$ such that $n-2<j$. Visually, the specialization of $M$ is
    \[
    M'=
    \begin{pmatrix}
    0 & x_{1,2} & \cdots & x_{1,n-2}
      & 0 & 0 & y_{1,1} & \cdots & y_{1,s}\\
    -x_{1,2} & 0 & \cdots & x_{2,n-2}
      & 0 & 0 & y_{2,1} & \cdots & y_{2,s}\\
    \vdots & \vdots & \ddots & \vdots
      & \vdots & \vdots & \vdots & & \vdots\\
    -x_{1,n-2} & -x_{2,n-2} & \cdots & 0
      & 0 & 0 & y_{n-2,1} & \cdots & y_{n-2,s}\\
    0 & 0 & \cdots & 0
      & 0 & 1 & 0 & \cdots & 0\\
    0 & 0 & \cdots & 0
      & -1 & 0 & 0 & \cdots & 0\\
    -y_{1,1} & -y_{2,1} & \cdots & -y_{n-2,1}
      & 0 & 0 & 0 & \cdots & 0\\
    \vdots & \vdots & & \vdots
      & \vdots & \vdots & \vdots & \ddots & \vdots\\
    -y_{1,s} & -y_{2,s} & \cdots & -y_{n-2,s}
      & 0 & 0 & 0 & \cdots & 0
    \end{pmatrix}.
    \]
    Recall that $I(n,s)$ is generated by the Pfaffians of $M$ of all sizes containing the top left $n\times n$ block, ignoring however the $n$-minors of $Y$. The corresponding Pfaffian generators of $M'$ are precisely the generators of $P(n-2,s)$, which in turn forces that $\ara(P(n-2,s))\leq\ara(I(n,s))$. Hence, by Corollary \ref{cor:araI(n,s)}, 
    \[
    d(n-2,s)=\ara(P(n-2,s))\leq\ara(I(n,s))\leq d(n-2,s).
    \]
    In characteristic zero, the assertion on cohomological dimension follows along similar lines.
\end{proof}

Finally, we note that Theorems \ref{thm:cdP(n,s)} and \ref{thm:ara-in-char-p} impose sharp upper bounds on the cohomological dimension and arithmetic rank of any residual intersection of a Gorenstein perfect ideal of height three in a Gorenstein local ring.  

\begin{cor}\label{cor:araRI}
Let $(A,\frakm)$ be a Gorenstein local ring. Let $I$ be a Gorenstein perfect ideal of $A$ of height three with $\mu(I)=n$. Given an integer $s\geq3$, let $J$ be an $s$-residual intersection of $I$. Then, 
\[
\lcd(J)\leq\ara(J)\leq \binom{s}{2}-\binom{s-n}{2}+1.
\]
Moreover, if $\chr A=0$ and $J$ is the generic $s$-residual intersection of $I$, then both the inequalities above are attained under any of the following assumptions:
\begin{enumerate}[leftmargin=*, align=left, label=(\alph*)]
    \item $A$ contains a field and $\height(\Fitt_{n-1}(I))=\binom{n}{2}$.
    \item $\chr(A/\frakm)=p>0$ and $\height(\Fitt_{n-1}(I)+(p))=\binom{n}{2}+1$.
\end{enumerate}
\end{cor}
\begin{proof}
    By \cite[Theorem 2.1]{BuchsbaumEisenbudAJM}, $\mu(I)=n$ is odd and $I$ is the ideal of submaximal Pfaffians of an alternating $n\times n$ matrix $\mathcal X$. Moreover, as $\mathcal X$ is the presentation matrix of $I$, $\Fitt_{n-1}(I)=I_1(\mathcal X)$. 
    
    By \cite[Theorem 10.5]{Kustin-Ulrich-Memoirs}, $J$ is the image of $P(n,s)$ under an evaluation map on the $X_n$ and $Y_{n,s}$ entries in $\ZZ[X_{n},Y_{n,s}]$. As arithmetic rank can only go down along ring homomorphisms, 
    \[
    \lcd(J)\leq\ara(J)\leq\ara(P(n,s))=\binom{s}{2}-\binom{s-n}{2}+1.
    \]
    Now, assume that $\chr A=0$ and $J$ is a generic $s$-residual intersection of $I$ in $A[Y_{n,s}]$. As $\Fitt_{n-1}(I)=I_1(\mathcal X)$, the assumption $\height(\Fitt_{n-1}(I))=\binom{n}{2}$ forces the entries of $\mathcal X$ to form a regular sequence. Similarly, $\height(\Fitt_{n-1}(I)+(p))=\binom{n}{2}+1$ asserts that the entries of $\mathcal X$ and $p$ form a regular sequence. 

    Hence, under the hypotheses of $(a)$, if $\KK$ is the field contained in $A$, we obtain a flat map 
    \[
    \Phi: \KK[X_n, Y_{n,s}] \to A[Y_{n,s}]
    \] sending entries of $X_n$ to corresponding entries of $\mathcal X$ (\cite{HartshorneFlat}). On the other hand, under the conditions imposed by $(b)$, the map 
    \[
    \Phi: \ZZ[X_n,Y_{n,s}]\into \ZZ_{(p)}[X_n,Y_{n,s}]\to A[Y_{n,s}]
    \] is flat.

    Let $A'=A[Y_{n,s}]_{(\frakm,Y_{n,s})}$ and $J'=J_{(\frakm,Y_{n,s})}$. Let $S$ be $\KK[X_n,Y_{n,s}]_{(X_n,Y_{n,s})}$ in case $(a)$ and $\ZZ[X_n,Y_{n,s}]_{(p,X_n,Y_{n,s})}$ in case $(b)$. 
    In either case, the localized map $\phi':S\to A'$ is faithfully flat. Let $P=P(n,s)$ and $d=d(n,s)$. Hence, we obtain that 
    \[
    H^d_{J'}(A')=H^d_{P(n,s)}(S)\otimes_S(A')\neq 0. \qedhere
    \]
\end{proof}

\section{Structure of the top local cohomology module} 
\label{sec:cohomology}

Adopting the notation of Setup \ref{setup}, we demonstrate how the techniques utilized in our arithmetic rank computations can be used to draw conclusions about the structure of the top nonvanishing local cohomology module of $R_{n,s}$ supported at $P(n,s)$. In particular, we prove that when $\KK=\ZZ$, the top local cohomology module has no integer torsion. We also compute the $D$-module length of said local cohomology module in characteristic $0$. 

We will need the following prime characteristic result, of independent interest, for our discussion. This result extends \cite[Theorem 4.3]{Pandey} which handles the case where $P(n,s)$ is the \textit{generic link} of submaximal Pfaffians, i.e., $s=3$.

\begin{thm}\label{thm:P(n,s)Freg}
$\KK[X_n,Y_{n,s}]/P(n,s)$ is a strongly $F$-regular ring
if $\KK$ is an $F$-finite field of positive characteristic, and has rational singularities if $\KK$ has characteristic zero.  
\end{thm}

\begin{proof}
It suffices to prove the positive characteristic assertion. It then follows that $\KK[X_n,Y_{n,s}]/P(n,s)$ is of $F$-rational type for $\KK$ of characteristic zero, and then by \cite[Theorem 4.3]{Smith}, it has rational singularities.

We proceed by induction on $n$. If $n=2$, by Remark \ref{rmk:n=2},
\[P(2,s)=(x_{12})+I_2 
\begin{pmatrix}
    y_{11} & y_{12} & \cdots & y_{1s} \\
    y_{21} & y_{22} & \cdots & y_{2s}
\end{pmatrix}.\]
As determinantal rings are strongly $F$-regular by \cite[\S 7]{Hochster-HunekeJAG}, $P(2,s)$ defines a strongly $F$-regular ring. We now fix $n\geq 3$. Invoking the induction hypothesis, it follows from Lemma \ref{lem:localization} that $y_{n,s}$ is a test element. Using Glassbrenner's criterion \cite{Glassbrenner}, it suffices to show that
\[y_{n,s}(P(n,s)^{[p]}:P(n,s)) \nsubseteq \frakm^{[p]}.\]
Using \cite[Corollary 3.3]{Pandey-Tarasova}, it further suffices to find an ideal $\fraka \subseteq P(n,s)$ generated by a regular sequence of length $\height(P(n,s))=s$ such that
\[y_{n,s}(\fraka ^{[p]}:\fraka) \nsubseteq \frakm^{[p]}.\]
We proceed to find such an ideal $\fraka$.

Consider the sequence $\{f_1,\ldots,f_s\}$ of $s$ elements in $P(n,s)$, where if $n$ is odd,
\[
f_i=
\begin{cases}
    [1,\ldots,i] & \text{if } i \text{ is odd and } i<n\\
    [2,\ldots,i] & \text{if } i \text{ is even and } i<n\\
    [i-n+1,\ldots,i] & \text{if } i\geq n,
\end{cases}
\]
and if $n$ is even,
\[
f_i=
\begin{cases}
    [\;] & \text{if } i=1\\
    [1,\ldots,i] & \text{if } i \text{ is even and } i< n\\
    [2,\ldots,i] & \text{if } i>1 \text{ is odd and } i<n\\
    [i-n+1,\ldots,i] & \text{if } i\geq n.
\end{cases}
\]
We prove that the elements $f_1,\ldots,f_s$ form a regular sequence by demonstrating the existence of a monomial order under which the initial monomials of $f_1,\ldots,f_s$ define a regular sequence.  Let $<$ be the lexicographic monomial order on $\KK[X_n,Y_{n,s}]$ induced by the ordering on variables
\begin{align*}
    &y_{1,s}>y_{2,s}>\cdots>y_{n,s}>y_{1,s-1}>y_{2,s-1}>\cdots>y_{n,2}>\\
    &\quad x_{1,n}>x_{1,n-1}>\cdots>x_{1,2}>x_{2,n}>x_{2,n-1}>\cdots>x_{n-1,n}>\\
    &\qquad y_{1,1}>y_{2,1}>\cdots>y_{n,1}.
\end{align*}
Under this monomial ordering, the initial monomials of the elements $f_1,\ldots,f_s$ are 
\begin{eqnarray*}
    \init_{<}[\;]&=&x_{1,n}x_{2,n-1}\cdots x_{\frac{n}{2},\frac{n+2}{2}},\\
    \init_{<}[1,\ldots,i]&=&y_{1,i}y_{2,i-1}\cdots y_{i-1,2}y_{\frac{i+n}{2},1}x_{i,n}x_{i+1,n-1}\cdots x_{\frac{i+n}{2}-1,\frac{i+n}{2}+1},\\
    \init_{<}[2,\ldots,i]&=&y_{1,i}y_{2,i-1}\cdots y_{i-1,2}x_{i,n}x_{i+1,n-1}\cdots x_{\frac{i+n-1}{2},\frac{i+n+1}{2}},\\
    \init_{<}[i-n+1,\ldots,i]&=&y_{1,i}y_{2,i-1}\cdots y_{n,i-n+1}.
\end{eqnarray*}
Observe that the variable $x_{a,b}$ occurs precisely when
$1\leq a+b-n\leq s$, and then only in the initial
monomial of $f_{a+b-n}$.

Similarly, for $d\neq1$, the variable $y_{c,d}$ does not appear in any initial monomial if $c+d>s+1$; otherwise, it appears only in the initial monomial of $f_{c+d-1}$. For $d=1$, the variable $y_{c,1}$ appears precisely when $1\leq2c-n\leq s$, and then only in the initial monomial of $f_{2c-n}$. Hence, the initial monomials $\init_<(f_1),\ldots,\init_<(f_s)$ involve pairwise disjoint sets of variables and therefore form a regular sequence. Consequently, $f_1,\ldots,f_s$ is a regular sequence.

Moreover, the variable $y_{n,s}$ does not appear in the initial term of any $f_i$. Let $\fraka=(f_1,\ldots,f_s).$ Then $(f_1f_2\cdots f_s)^{p-1}\in (\fraka^{[p]}:\fraka)$, but 
\[
y_{n,s}(f_1f_2\cdots f_s)^{p-1}\notin\frakm^{[p]},
\]
by the above argument since the initial terms of $f_i$ form a regular sequence of monomials and $y_{n,s}$ does not divide any of these initial monomials.
\end{proof}

In view of \cite[Theorem 1.3]{JPSW}, Theorem \ref{thm:P(n,s)Freg} immediately yields:

\begin{cor}\label{cor:integertorsion}
Let $\ZZ$ be the integers, $R=\ZZ[X_n,Y_{n,s}]$ and set $P:=P(n,s)$, an ideal of $R$. The local cohomology module $H^{d(n,s)}_P(R)$ is a $\QQ$-vector space. In particular, it has no integer torsion. 
\end{cor}

\begin{thm}\label{thm:Dmodlen}
Suppose $\KK$ is a field of characteristic zero in Setup \ref{setup}. Set $P=P(n,s)$ and $d=d(n,s)$.
The $D$-module length of $H^d_P(R)$ is as follows: 
\[
\ell_D(H^d_P(R)) = 
\begin{cases} 
1 & \text{if } s > n, \\ 
2 & \text{if } s = n. 
\end{cases}
\]
\end{thm}
\begin{proof}
If $s>n$ we know from Case 1 in the proof of Theorem \ref{thm:cdP(n,s)} that $H^d_P(R)\cong H^d_{I_n(Y)+(X)}(R)$. Since the variables in the standard generators of $(X)$ are disjoint from those in the standard generators of $I_n(Y)$ and since $d=\lcd(I_n(Y)+(X))$, a suitable Grothendieck spectral sequence shows that $H^d_P(R)=H^{d'}_{I_n(Y)}(R)\otimes_R H^{d''}_{(X)}(R)$ where $d'=\lcd(I_n(Y))$ and $d''=\lcd((X))$. By Kashiwara equivalence  then,    the length of the $D_R$-module $H^d_P(R)$ is equal to that of the $D_{R'}$-module $H^{d'}_{I_n(Y)}(R')$ where $R'=\KK[Y]$. However, for the ideal of maximal minors of a non-square matrix, it is well-known that the top local cohomology module has $D$-length 1; this seems to have appeared in print first in  \cite{Witt2012}.

If $n=s$, Case 1 of the proof of Theorem \ref{thm:cdP(n,s)} yields that $H^d_P(R)$ is a quotient of $H^d_{I_n(Y)+(X)}(R)$. Now, however, $Y$ is a square matrix with unique maximal minor $\Delta$ and $\lcd(I_n(Y))=1$. Moreover, $H^1_{I_n(Y)}(\KK[Y])$ has a \emph{unique} composition chain 
\[
0=M_0\subsetneq M_1\subsetneq \ldots\subsetneq M_n=H^1_{I_n(Y)}(\KK[Y]),
\]
where $M_i$ is generated by the class of $1/\Delta^{i}$ and has support equal to the variety of $I_n(Y)$, and where each quotient $M_i/M_{i-1}$ is an irreducible $D$-module with support equal to the variety of the ideal $I_{n-i+1}(Y)$; see \cite{RWW}.

By Lemma \ref{lem:localization} and general properties of local cohomology, inverting $y_{n,s}$ on $H^d_P(R)$ yields $H^d_{P'(n-1,s)}(R[y_{n,s}^{-1}])$ with notation as in that lemma. For $n=s$, the quantities $d(n,s)$ and $d(n-1,s)$ agree, and so $d=\lcd(P'(n-1,s))$. It follows that $H^d_P(R)\otimes_R R_{y_{n,s}}$ is not zero and thus $H^d_P(R)\neq M_n/M_{n-1}$.

Now consider a 2-minor of $Y$ that involves $y_{n,s}$, such as $[i,n\mid j,s]_Y= y_{i,j}y_{n,s}-y_{i,s}y_{n,j}$, and suppose we were to invert first $y_{n,s}$ and then $[i,n\mid j,s]_Y$.  In the presence of invertible $y_{n,s}$, $[i,n\mid j,s]_Y$ equals $y'_{i,j}$ in Lemma \ref{lem:localization} up to multiplication by a unit, and so $H^d_P(R)[1/y_{n,s},1/[i,n\mid j,s]_Y]$ is isomorphic to a tensor product with $H^d_{P(n-2,s)}(R)$.  But for $n=s$, $d(n-1,s)=d(n-2,s)+1$, and so this second localization kills $H^d_P(R)$. 
It follows, that no composition factor of $H^d_P(R)$ can contain any $M_i/M_{i-1}$ for $i<n-1$, and so the structure of $H^1_{I_n(Y)}(R)$ forces that $H^d_P(R)=M_n/M_{n-2}$ of length $2$.
\end{proof}

\section{Residual intersections that are set-theoretic complete intersections}
\label{sec:STCI}

In this section, we construct residual intersections of complete intersection ideals as well as those of height three perfect Gorenstein ideals which turn out to be set-theoretic complete intersections. These ideals arise as careful linear specializations of the corresponding generic residual intersections; see \cite{Valla} for related constructions.

In \cite[Theorem 4.9]{BMSMP}, a sharp upper bound on the arithmetic rank of the residual intersections of a complete intersection ideal is obtained over any Noetherian local ring. Moreover, the result shows that when such residual intersections---in a polynomial extension of the base ring---are generic, they fail to be set-theoretic complete intersections in characteristic zero.  It is then natural to search for a sharp lower bound for the arithmetic ranks among all residual intersections of a complete intersection ideal in some polynomial extension of the base ring. In this direction, we prove:

\begin{thm}\label{thm:STCIforBMSMP}
Choose $s\geq n> 1$ and let $I$ be an ideal generated by a regular sequence of $n$ elements in a Gorenstein local ring $(A,\frakm)$. There exists an $s$-residual intersection of $IB$, in a polynomial extension $B$ of $A$, which is a set-theoretic complete intersection.
\end{thm}
\begin{proof}
We first construct such a residual intersection for an ideal of $n$ variables in a polynomial ring over the integers. Let $X=(x_{i,j})$ and $Y=(y_i)$ be matrices of variables of sizes $s\times n$ and $n\times1$ respectively. Consider the polynomial ring $\ZZ[X,Y]$. Then, the ideal $(XY):(Y)$ is an $s$-residual intersection of $(Y)$. By \cite[Example 3.4]{HunekeUlrichRI},
\[
J\colonequals (XY):(Y)=(XY)+I_n(X).
\]
As $J$ is an $s$-residual intersection, $\height(J)=s$. The arithmetic rank of $J$ was computed in \cite[Theorem 4.7]{BMSMP} to be $\ara(J)=n(s-n+1)+1>s=\height(J)$. We now produce a non-trivial linear specialization of $J$ which is a set-theoretic complete intersection.

Define new variables $a_n,\ldots,a_s$ and set
\begin{align*}
x_{i,j}=
\begin{cases}
0 & \text{if } i+j>s+1,\\
a_{i+j-1} & \text{if } n+1\leq i+j\leq s+1.
\end{cases}
\end{align*}
We denote this specialization of $X$ as $X'$. Visually, the variables in each anti-diagonal are set to be equal starting from the anti-diagonal containing the last entry of the first row all the way to the anti-diagonal containing the first entry of the last row, while all lower anti-diagonal variables are set to be zero. Schematically,
\[
X'=
\begin{pmatrix}
x_{1,1} & x_{1,2} & \cdots & x_{1,n-1} & a_n\\
x_{2,1} & x_{2,2} & \cdots & a_n & a_{n+1}\\
\vdots & \vdots & \iddots & \iddots & \vdots\\
x_{n-1,1} & a_n & \iddots & \cdots & a_{2n-2}\\
a_n & a_{n+1} & \cdots & a_{2n-2} & a_{2n-1}\\
\vdots & \vdots & & \vdots & \vdots\\
a_{s-n+1} & a_{s-n+2} & \cdots & a_{s-1} & a_s\\
a_{s-n+2} & a_{s-n+3} & \cdots & a_s & 0\\
\vdots & \iddots & \iddots & \vdots & \vdots\\
a_{s-1} & a_s & \cdots & 0 & 0\\
a_s & 0 & \cdots & 0 & 0
\end{pmatrix}_{s\times n}.
\]
where we use the convention that $a_k=0$ if $k>s$.

We shall show that the ideal $J'\colonequals(X'Y)+I_n(X')$ is a linear specialization of $J$ and that it is a set-theoretic complete intersection.

We begin by showing that the matrix specialization of $X$ to $X'$ corresponds to a regular sequence in $\ZZ[X]$. The matrix specialization involves going modulo the sequence of elements
\begin{align*}
&x_{s,n},\;x_{s,n-1},\;x_{s-1,n},\;x_{s,n-2},
\;x_{s-1,n-1},\;x_{s-2,n},\;\ldots,\;x_{s,2},\ldots,\; x_{s-n+2,n},\\
&x_{s,1}-x_{s-1,2},\;x_{s-1,2}-x_{s-2,3},
\;\ldots,\;x_{s-n+2,n-1}-x_{s-n+1,n},\\
&x_{s-1,1}-x_{s-2,2},\;\ldots,\;x_{2,n-1}-x_{1,n}.
\end{align*}
Under the lexicographic monomial order induced by
\[
x_{1,1}<x_{1,2}<\cdots<x_{1,n}
<x_{2,1}<\cdots<x_{2,n}<\cdots<x_{s,n},
\]
each term in the above sequence has a distinct initial monomial. Hence, the above sequence is regular. We now find the set-theoretic generators of $I_n(X')$.

\noindent \textbf{\emph{Claim}}:
$\sqrt{I_n(X')}=(a_n,\ldots,a_s)$.

\noindent \emph{Proof of claim}: The containment $I_n(X')\subseteq(a_n,\ldots,a_s)$ is clear as any $n$ distinct rows of $X'$ must have among them at least one row with only $a_k$ and $0$ entries. It remains to prove that $a_k\in\sqrt{I_n(X')}$ for all $k\in\{n,\ldots,s\}$. Observe that the minor $[s-n+1,\ldots,s]$ obtained by choosing the last $n$ rows of $X'$ evaluates to $a_s^n$, which implies that $a_s\in\sqrt{I_n(X')}$. Next, choose the minor $[s-n,\ldots,s-1]$ and evaluate along the last row to obtain
\[
[s-n,\ldots,s-1]=a_{s-1}^n+a_sf,
\]
for some $f\in\ZZ[X',Y]$. Hence, $a_{s-1}\in\sqrt{I_n(X')}$. Proceeding inductively, if we assume that $a_s,a_{s-1},\ldots,a_{k+1}\in\sqrt{I_n(X')}$, then evaluation of the minor $[k-n+1,\ldots,k]$ along the last row yields
\[
[k-n+1,\ldots,k]
=\pm a_k^n+a_{k+1}f_1+\cdots+a_{k+n-1}f_{n-1},
\]
for some $f_i\in\ZZ[X',Y]$. Hence, $a_k\in\sqrt{I_n(X')}$ for all $k$ and the claim follows.

\smallskip

We now return to the proof of the theorem. The claim above asserts that
\[
\sqrt{J'}=\sqrt{(X'Y)+(a_n,\ldots,a_s)}.
\]
Observe that the entries in the last $s-n+1$ rows lie in the ideal $(a_n,\ldots,a_s)$, and define  $(g_1,\ldots,g_s)\coloneqq (X'Y)^{\text{tr}}$, the transpose.  Then, the ideal $(X'Y)+(a_n,\ldots,a_s)$ is generated by $g_1,\ldots,g_{n-1},a_n,\ldots,a_s$. It remains to prove that $\height(J')=s$.

Reducing the elements $g_1,\ldots,g_{n-1}$ modulo the ideal $(a_n,\ldots,a_s)$ gives us the elements $h_1,\ldots,h_{n-1}$ defined as follows:
\[
h_i=\sum_{j=1}^{n-i}x_{i,j}y_j.
\]
Clearly, the set $\Gamma=\{h_1,\ldots,h_{n-1},a_n,\ldots,a_s\}$ of $s$ elements also forms a set-theoretic generating set of $J'$. Observe that these polynomials form a homogeneous regular sequence in $\ZZ[X',Y]$. Indeed, define the lexicographic monomial order on $\ZZ[X',Y]$ induced from
\[
y_n>y_{n-1}>\cdots>y_1
\quad\text{and}\quad
y_k>x_{i,j}\quad\text{for all }i,j,k,
\]
where the order between the variables $x_{i,j}$ is arbitrary. Then the initial monomial of $h_i$ is $x_{i,n-i}y_{n-i}$, and thus  the elements of $\Gamma$ have initial monomials involving  disjoint sets of  variables. Hence, the set-theoretic complete intersection ideal $J'$ satisfies $\height(J')=s=\height(J)$ and so $J'$ is a linear specialization of $J$ by Proposition \ref{prop:specialization} since the generic $s$-residual intersection $\ZZ[X,Y]/J$ is Cohen--Macaulay.

Finally, let $\underline f:=(f_1,\ldots,f_n)^{\tr}$ be a regular sequence in the Gorenstein local ring $A$ arranged as a column vector. Define
\[
\varphi:\ZZ[y_1,\ldots,y_n][X']\longrightarrow A[X']
\]
with $\varphi(y_i)=f_i$ and each variable of $X'$ mapping to itself. Set
\[
J'' :=(X'\underline f):(\underline f)
=I_n(X')+(X'\underline f),
\]
an $s$-residual intersection of the ideal $(\underline f)A[X']$, by the structure theorem of residual intersections of a complete intersection \cite[Example 3.4, Theorem 4.7]{HunekeUlrichRI}. Since arithmetic rank can only go down along a ring homomorphism, we have
\[
s\leq\ara(J'')\leq\ara(J')=\height(J')=s,
\]
so that $\ara(J'')=s$. Thus $J''$ is a set-theoretic complete intersection.
\end{proof}

The arguments in the above proof motivate our next result. We can pass from the generic residual intersection of a generic height three Gorenstein ideal to a specific residual intersection of a particular height three Gorenstein ideal by carefully constructing our linear specialization to ensure that the target is a set-theoretic complete intersection. Note that such a specialization trivially exists when the ground field is infinite: take the specialized ideal to be $\frakm$-primary.  

We now construct a nontrivial linear specialization of the generic residual intersection $P(n,s)$ which is a set-theoretic complete intersection. Recall that by Theorems \ref{thm:cdP(n,s)} and \ref{thm:ara-in-char-p}, $P(n,s)$ is itself never a set-theoretic complete intersection. 

\begin{thm}\label{thm:STCIforGor}
Let $\KK$ be a field, $X$ an odd-sized generic alternating matrix of size $n\geq 3$, and $I$ be the ideal in $\KK[X]$ generated by the submaximal Pfaffians of $X$. Recall that $I$ is a generic Gorenstein ideal of height three. Let $P$ be the generic $s$-residual intersection of $I$.

$P$ can be linearly specialized to a set-theoretic complete intersection $P'$ such that $P'$ is an $s$-residual intersection of a specialization $I'$ of $I$. Moreover, $I'$ remains a Gorenstein ideal of height three with the same number of generators as $I$. 
\end{thm}
\begin{proof}
    Let $n\geq3$ be an odd integer. Let $X$, $Y$, $M$, $R$ and $P\colonequals P(n,s)$ be as in Setup \ref{setup}. As $P$ is an $s$-residual intersection, $\height(P)=s$. By Theorems \ref{thm:cdP(n,s)} and \ref{thm:ara-in-char-p}, the arithmetic rank of $P$ is $\binom{s}{2}-\binom{s-n}{2}+1$. 

    Define new variables $a_1,\ldots,a_s$, and set
    \begin{align*}
    &x_{i,j}=
    \begin{cases}
        0 & \text{if } i+j-n>s\\
        a_{i+j-n} & \text{if } 1\leq i+j-n\leq s,
    \end{cases}\\
    &y_{i,j}=
    \begin{cases}
        0 & \text{if } j\geq2 \text{ and } i+j-1>s\\
        a_{i+j-1} & \text{if } j\geq 2 \text{ and } i+j-1\leq s\\
        0 & \text{if } j=1 \text{ and } 2i-n>s\\
        a_{2i-n} & \text{if } j=1 \text{ and } 1\leq 2i-n\leq s.
    \end{cases}
    \end{align*}
    Denote the corresponding specializations of $X$, $Y$, and $M$ as $X'$, $Y'$, and $M'$ respectively. Visually, for $n=5$ and $s=7$, 
    \[
    X'=
    \begin{pmatrix}
    0&x_{12}&x_{13}&x_{14}&a_1\\
    -x_{12}&0&x_{23}&a_1&a_2\\
    -x_{13}&-x_{23}&0&a_2&a_3\\
    -x_{14}&-a_1&-a_2&0&a_4\\
    -a_1&-a_2&-a_3&-a_4&0
    \end{pmatrix},
    \qquad
    Y'=
    \begin{pmatrix}
    y_{1,1}&a_2&a_3&a_4&a_5&a_6&a_7\\
    y_{2,1}&a_3&a_4&a_5&a_6&a_7&0\\
    a_1&a_4&a_5&a_6&a_7&0&0\\
    a_3&a_5&a_6&a_7&0&0&0\\
    a_5&a_6&a_7&0&0&0&0
    \end{pmatrix}.
    \]
    For $n=5$ and $s=3$, 
    \[
    X'=
    \begin{pmatrix}
    0&x_{12}&x_{13}&x_{14}&a_1\\
    -x_{12}&0&x_{23}&a_1&a_2\\
    -x_{13}&-x_{23}&0&a_2&a_3\\
    -x_{14}&-a_1&-a_2&0&0\\
    -a_1&-a_2&-a_3&0&0
    \end{pmatrix},
    \qquad
    Y'=
    \begin{pmatrix}
    y_{1,1}&a_2&a_3\\
    y_{2,1}&a_3&0\\
    a_1&0&0\\
    a_3&0&0\\
    0&0&0
    \end{pmatrix}.
    \]
    Employing techniques similar to the proof of Theorem \ref{thm:STCIforBMSMP}, one may observe that the matrix specializations above correspond to going modulo a linear regular sequence in $R=\KK[X,Y]$.
    Recall that
    \[
    M'=\begin{pmatrix}
        X' & Y'\\
        -Y'^{\tr} & 0
    \end{pmatrix}.
    \]
    Define $P'$ to be the ideal of $R'\colonequals \KK[X',Y']$ generated by Pfaffians of all sizes of $M'$ containing the top left block $X'$. We shall show that $P'$ is a linear specialization of $P$ and that it is a set-theoretic complete intersection. 

    We begin by showing that $\{a_1,\ldots,a_s\}$ is a set-theoretic generating set of $P'$. If $s\geq n$, then $[s-n+1,\ldots,s]=a_s^n$. If $s<n$ and $s$ is odd, then $[1,\ldots,s]=\pm a_s^{\frac{n+s}{2}}$. If $s<n$ and $s$ is even, then $[2,\ldots,s]=\pm a_s^{\frac{n+s-1}{2}}$. Hence, $a_s\in\sqrt{P'}$. Now, suppose we have already shown that $a_s,\ldots,a_{k+1}\in\sqrt{P'}$. Then, if $k\geq n$, by Laplace expansion of the Pfaffian along the $k^{\text{th}}$ column of $Y$,
    \[
    [k-n+1,\ldots,k] =\pm a_k^n+f_{k+1}a_{k+1}+\cdots+f_sa_s,
    \]
    for $f_i\in \KK[X',Y']$. Similarly, if $k<n$, we obtain similar expressions for $[1,\ldots,k]$ if $k$ is odd and for $[2,\ldots,k]$ if $k$ is even. In all cases, we can conclude that $a_k\in\sqrt{P'}$. By induction, $(a_1,\ldots,a_s)\subseteq\sqrt{P'}$. 

    To show the other inclusion, observe that the only generator of $P'$ which is not obviously in $(a_1,\ldots,a_s)$ is $[1]$, as every entry of $Y$ from the second column onwards is in $(a_1,\ldots,a_s)$. However, 
    \[
    [1]=y_{1,1}\text{Pf}_{2,3,\ldots,n}(X')-y_{2,1}\text{Pf}_{1,3,\ldots,n}(X')+\cdots\pm y_{\frac{n-1}{2},1}\text{Pf}_{1,\ldots,\widehat{\frac{n-1}{2}},\ldots,n}(X')+F,
    \]
    where $F\in(a_1,\ldots,a_s)$ and $\text{Pf}_{i_1,\ldots,i_t}(X)$ is the Pfaffian of the alternating submatrix of $X$ formed by the rows and columns numbered $i_1,\ldots,i_t$. Each of the Pfaffians appearing in the above expression involves submatrices with `too many' entries in $(a_1,\ldots,a_s)$. More specifically, all the entries on and below the antidiagonal of these submatrices are in $(a_1,\ldots,a_s)$. Therefore, $[1]\in (a_1,\ldots,a_s)$. 

    We have thus proved that $\sqrt{P'}=(a_1,\ldots,a_s)$, which leads us to conclude that $P'$ is a set-theoretic complete intersection. As $\height(P')=\height(P)=s$ and $R/P$ is Cohen--Macaulay, by Proposition \ref{prop:specialization}, it follows that $P'$ is a specialization of $P$.

    Finally, we show that $P'$ is a residual intersection of a perfect Gorenstein ideal of height $3$. To do so, we first prove that the ideal $I'$ of submaximal Pfaffians of $X'$ is a Gorenstein ideal of height $3$. Under any graded reverse lexicographic monomial order agreeing with the partial ordering of variables
    \begin{align*}
        a_k<a_l &\text{ if } k>l,\\
        a_k<x_{i,j} &\text{ if } k>i+j-n,\\
        x_{i,j}<x_{k,l}&\text{ if } i+j>k+l,
    \end{align*}
    consider the submaximal Pfaffians $[1,\ldots,n-1]$, $[1,\ldots,\frac{n-1}{2},\frac{n+3}{2},\ldots,n]$, and $[2,\ldots,n]$. Observe that
    \begin{eqnarray*}
    \init[1,\ldots,n-1]&=&x_{1,n-1}x_{2,n-2}\cdots x_{\frac{n-1}{2},\frac{n+1}{2}},\\
    \init[1,\ldots,\frac{n-1}{2},\frac{n+3}{2},\ldots,n]&=&a_1^{\frac{n-1}{2}},\\
    \init[2,\ldots,n]&=&a_2^{\frac{n-1}{2}}.
    \end{eqnarray*}
    As the initial terms of these submaximal Pfaffians form a regular sequence, so do the Pfaffians themselves. Thus the height of $I'$ is at least three, and since the ideal $I$ of submaximal Pfaffians of a generic alternating matrix has height three, the height of $I'$ cannot be any greater. Hence, by Proposition \ref{prop:specialization}, $I'$ is a specialization of $I$. The Gorenstein property is preserved under specialization. The number of generators of $I'$ is $n$ as every element of $X'$ is in the homogeneous maximal ideal of $R'$ and the height of $I'$ is three. By the structure theorem of residual intersections of height three Gorenstein ideals \cite[Theorem 10.5]{Kustin-Ulrich-Memoirs}, $P'$ is an $s$-residual intersection of $I'$. 
    \end{proof}

\section{A question on arithmetic rank}
\label{sec:questions}

In Section \ref{sec:genericRI}, we observed that for the ideal $P(n,n)=I(n,n)+(\det(Y))$, we have $\ara(P(n,n))=\ara(I(n,n))+1$. More generally, one can study how the arithmetic rank changes under adding a regular element.  

We ask the following natural question about the behaviour of arithmetic rank: 

\begin{question}\label{question:ara}
Let $\KK$ be a field, $R\coloneqq \KK[x_1,\ldots,x_t]$, and $I$ a proper ideal of $R$. Let $z$ be a new variable; set $S\coloneqq R[z]$. Is
\[\ara(IS+(z)) = \ara(I)+1?\qedhere\]
\end{question}
We start with a number of remarks on Question \ref{question:ara}.

\begin{rmk}\label{rmk:araretract}
In the setting of Question \ref{question:ara}, note that $\ara_S(IS)=\ara_R(I)$ as substituting $z=0$ in the set-theoretic generators of $IS$ in $S$ produces set-theoretic generators of $I$ in $R$. In fact, whenever an inclusion of rings $R \subseteq S$ admits an algebra retract, then $\ara(IS)=\ara(I)$ for any ideal $I$ of $R$. 

Additionally, substituting $z=0$ in the set-theoretic generators of $IS+(z)$ also yields
\[
\ara(I)\leq\ara(IS+(z))\leq\ara(I)+1. \qedhere
\]

\end{rmk}

\begin{rmk}
    In Question \ref{question:ara}, weaken the assumption on $R$ and consider it to be a Unique Factorization Domain (UFD). If we have $\ara(I)=1$, then indeed $\ara(IS+(z))=2$. For, if $\ara(IS+(z))=1$, suppose $f\in S$ generates $IS+(z)$ up to radical. Then $z^k\in(f)$ for some $k\in\NN$, but then $f$ is also a power of $z$ since $S=R[z]$ is also a UFD. Substituting $z=0$ then forces $I=0$, a contradiction.  
\end{rmk}

\begin{exa}
In Question \ref{question:ara}, if one only assumes that $R$ is a regular ring of finite Krull dimension containing a field, then the answer is in the \textit{negative}. The following example was pointed out to us by Bhargav Bhatt. Consider the ring
\[
R=\QQ[x,y]/(y^2-x^3+2),
\] 
$I=(x-3,y-5)\subseteq R$, and set $S=R[z]$. Observe that $R$ is a regular ring. We claim 
    \[
    \ara(I)=\ara(IS+(z))=2.
    \]
To see this, note that $y^2=x^3-2$ is the Weierstrass equation of the elliptic curve $Y^2Z=X^3-2Z^3$. By the Nagell--Lutz Theorem \cite[Corollary VIII.7.2]{Silverman}, the point $(3,5)$ does not have finite order in the elliptic curve. Hence, the height one prime ideal $I$ is not torsion in the divisor class group $\text{Cl}(R)$. If $\sqrt{(f)}=I$ for some $f\in R$, then, since $R$ is a Dedekind domain, $(f)=I^m$ for some $m\geq1$. This would make $[I]$ torsion in $\text{Cl}(R)$, a contradiction. This implies that $I$ cannot be singly generated up to radical and thus $\ara(I)=2$.

On the other hand, since $S$ is a polynomial ring over $R$ of Krull dimension $2$, by \cite[Theorem 1]{EisenbudEvans}, we get that $\ara(IS+(z))\le 2$, and the inequality cannot be strict since a purported single generator would need to (have a multiple that is a power of $z$ and hence) be a power of $z$, up to a unit.
\end{exa}

\begin{rmk}
    In Question \ref{question:ara}, if we only assume that $z$ is an element of $R$ that is regular on $R/I$, then while $\ara(I+(z))\leq\ara(I)+1$, this upper bound can be made \textit{arbitrarily weak}, that is, the arithmetic rank can fall by an arbitrarily large quantity when adding a regular element to the ideal.
    Indeed, let $t\geq1$, $R=\KK[x_1,x_2,z,y_1,\ldots,y_{2t}]$, and 
    \[
    I=(x_1z+x_2^2,x_1+x_2,x_1y_1+x_2y_2,\ldots,x_1y_{2t-1}+x_2y_{2t}).
    \]
    Note that the minimal primes of $I$ are $(x_1,x_2)$ and 
    \[
    (x_1+z,x_1+x_2,y_1-y_2,\ldots,y_{2t-1}-y_{2t}).
    \]
    Moreover, $I$ is radical because 
    \[
    I=(x_1,x_2)\cap (x_1+z,x_1+x_2,y_1-y_2,\ldots,y_{2t-1}-y_{2t}).
    \]
    As $I$ has a minimal prime of height $t+2$, the cohomological dimension of $I$ is at least $t+2$. Since $I$ has $t+2$ generators, $\lcd(I)=\ara(I)=t+2.$ Note that $z$ is a regular element on $R/I$ as it does not belong to either of the minimal primes. However, 
    \begin{align*}
    \sqrt{I+(z)}
    &=\sqrt{(z,x_2^2,x_1+x_2,x_1y_1+x_2y_2,\ldots,x_1y_{2t-1}+x_2y_{2t})}\\
    &=\sqrt{(z,x_2,x_1)}.
    \end{align*}
    Hence, $\ara(I+(z))=3$. 
\end{rmk}

\begin{rmk}\label{rmk:ara-realized-by-local-cohomology}
If $\ara(I)$ is realized by a local cohomology obstruction, that is,
$$H^{\ara(I)}_I(R) \neq 0,$$
then Question \ref{question:ara} clearly has an affirmative answer:
\[H^{\ara(I)+1}_{IS+(z)}(R[z]) \cong H^{\ara(I)}_I(R)\otimes_{\KK} H^1_{(z)}(\KK[z])=\bigoplus_{i\geq 1}H^{\ara(I)}_I(R)z^{-i} \neq 0.\]
This holds, for instance, when the ideal $I$ is a set-theoretic complete intersection. The content of the next proposition is that Question \ref{question:ara} also has a positive answer when the arithmetic rank of $I$ is realized by an \'etale cohomological obstruction.
\end{rmk}

\begin{prp}\label{prop:ara-realized-by-etale-cohomology}
    Let $\KK$ be an algebraically closed field, and $R$, $I$, and $S$ be as in Question \ref{question:ara}. Suppose $\ara(I)=n>1$. Let $U=\KK^t\setminus V(I)$ and $W=\KK^{t+1}\setminus V(IS+(z))$. Suppose $H_{\et}^{t+n-1}(U,G)\neq 0$ for some $\ell$-torsion finite Abelian group $G$ with $\ell$ invertible in $\KK$. Then $H_\et^{t+n+1}(W,G)\neq 0$. Hence, $\ara(IS+(z))=n+1$. 
\end{prp}
\begin{proof}
    Let $W_1=\KK^{t+1}\setminus V(IS)$ and $W_2=\KK^{t+1}\setminus V(z)$. Then $W=W_1\cup W_2$ and
    \[
    W_{12}\colonequals W_1\cap W_2=\KK^{t+1}\setminus V(z\cdot IS).
    \]
    By Proposition \ref{prop:ara-from-et}, $H_\et^i(U,G)=0$ for $i>t+n-1$. Hence, $H_\et^i(W_1,G)=0$ for $i>t+n-1$. Also by Proposition \ref{prop:ara-from-et}, $H^i_\et(W_2,G)=0$ for $i>t+1$. Therefore, by Proposition \ref{prop:et-mv},
    \[
    H_\et^{t+n+1}(W,G)\cong H_\et^{t+n}(W_{12},G).
    \]
    In order to prove that $H_\et^{t+n+1}(W,G)\neq 0$, it suffices to prove that $H_\et^{t+n}(W_{12},G)\neq0$. 

    Let $V=V(z\cdot IS)$. Observe that 
    \[
    V=(\KK^t\times\{0\})\sqcup (V(I)\times\KK^*). 
    \]
    As $V(I)\times\KK^*=(\KK^t\times \KK^*)\cap V$ is open in $V$, Proposition \ref{prop:les-triple} applies. Note that $\cptdim(\KK^t)=2t$. By Poincar\'e duality (Proposition \ref{prop:Poincareduality}), $\cptdim(U)=2t-(t+n-1)=t-n+1$. By Proposition \ref{prop:les-triple}, $\cptdim(V(I))=t-n$. Further, 
    \begin{align*}
    \cptdim(V(I)\times \KK^*)
    &=\cptdim(V(I))+\cptdim(\KK^*)\\
    &=t-n+1.
    \end{align*}
    By Proposition \ref{prop:les-triple}, we can now conclude that $\cptdim(V)=t-n+1$. As $V\sqcup W_{12}=\KK^{t+1}$, again by Proposition \ref{prop:les-triple}, $\cptdim(W_{12})=t-n+2$. Finally, by Poincar\'e duality, $H_\et^{t+n}(W_{12},G)\neq 0$. The last assertion follows from Proposition \ref{prop:ara-from-et}.
\end{proof}

Remark \ref{rmk:ara-realized-by-local-cohomology} and Proposition \ref{prop:ara-realized-by-etale-cohomology} indicate that the answer to Question \ref{question:ara} is positive for several families of ideals where arithmetic ranks have been computed in the literature \cite{BrunsSchwanzl, Barile, JPSW, BMSMP, BMMS}, including this article. In each of these cases, the lower bound is detected by a local or \'etale cohomological obstruction. 

We pose an alternate question. 

\begin{question}\label{question:alternate}
    Suppose $R\coloneqq \KK[x_1,\ldots,x_t]$, $S\coloneqq R[z]$ and $(f_1,\ldots,f_n)$ is a proper ideal of $S$ such that $z\in\sqrt{(f_1,f_2,\ldots,f_n)}$. For $f\in S$, let $f(0)$ denote the element of $R$ obtained by substituting $z=0$ in $f$. Is 
    \[
    \ara(f_1(0),\ldots,f_n(0))<n?\qedhere
    \]
\end{question}

\begin{rmk}\label{rmk:equivalence}
Questions \ref{question:ara} and \ref{question:alternate} are equivalent, as any counterexample to one question also provides a counterexample to the other. Indeed, suppose there exists an ideal $I\subsetneq R$ such that $\ara(IS+(z))=\ara(I)=n$. Let $f_1,\ldots,f_n$ be the minimal set-theoretic generators of $IS+(z)$. Then $z\in \sqrt{(f_1,\ldots,f_n)}$ and $\sqrt{I}=\sqrt{(f_1(0),\ldots,f_n(0))}$. Thus $\ara(f_1(0),\ldots,f_n(0))=\ara(I)=n.$ 

Conversely, suppose there exists an ideal $(f_1,\ldots,f_n)\subsetneq S$ such that $z\in \sqrt{(f_1,\ldots,f_n)}$ and $\ara(f_1(0),\ldots,f_n(0))=n$. Then, let $I=(f_1(0),\ldots,f_n(0))\subseteq R$. As $\sqrt{IS+(z)}=\sqrt{(f_1,\ldots,f_n)}$, we have
$\ara(IS+(z))=\ara(I)=n$.
\end{rmk}

\begin{rmk}\label{rmk:lcd}
    In the setting of Question \ref{question:alternate}, observe that $\lcd(f_1(0),\ldots,f_n(0))<n$. This is because if $\lcd(f_1(0),\ldots,f_n(0))=n$, then by the argument of Remark \ref{rmk:ara-realized-by-local-cohomology}, 
    \[
    \lcd(f_1(0),\ldots,f_n(0),z)=n+1,
    \]
    which contradicts $\sqrt{(f_1,\ldots,f_n)}=\sqrt{(f_1(0),\ldots,f_n(0),z)}$.
\end{rmk}

Remark \ref{rmk:lcd} suggests a potentially novel technique for drawing conclusions about cohomological dimension, as illustrated by the following examples.

\begin{exa}\label{exa:hypersurface}
    Consider the hypersurface 
    \[
    R\coloneqq \frac{\KK[h_1,h_2,h_3,h_4,a,b,c]}{(ah_1+h_1h_2+a^2h_3+bh_1h_3-ch_4)},
    \] and let $I=(h_1,h_2,h_3,h_4)\subseteq R$. As $h_1,h_2,h_3$ is a maximal regular sequence in $I$, $\height(I)=3$. In the ring $R[z]$, 
    \[
    z^5=(z^2-a)(z^3+az+h_1)-h_1(z^2+bz+h_2)+(a^2+h_1b)(z-h_3)+ch_4,
    \]
    which in turn implies
    \[
    \sqrt{IR[z]+(z)}=\sqrt{(z^3+az+h_1,z^2+bz+h_2,z-h_3,h_4)}.
    \]
    Hence, $\ara(IR[z]+(z))\leq 4$, which forces $\lcd(I)\leq 3$. As $I$ has height $3$, we conclude that  $\lcd(I)=3$.
\end{exa}

\begin{exa}\label{exa:general}
    Let $n\geq1$, $d\geq0$, $e\geq0$, and $1\leq t\leq d+e$ be integers. Consider
    \[
    R\coloneqq \frac{\KK[f_{i,j},g_{i,k}\mid 1\leq i\leq n,\; 0\leq j \leq d,\; 0\leq k\leq e]}{\bigg(\sum_{i=1}^n \sum_{j+k=s}f_{i,j}g_{i,k}- \delta_{s,t}\mid s=0,\ldots,d+e\bigg)}.
    \]
    Let $I=(f_{i,0}\mid 1\leq i \leq n)\subseteq R$. Observe that in $R[z]$, 
    \[
    z^t=\sum_{i=1}^n\bigg(\sum_{j=0}^d f_{i,j}z^j\bigg)\bigg(\sum_{k=0}^e g_{i,k}z^k\bigg).
    \]
    Hence, $z\in \sqrt{\big(\sum_{j=0}^d f_{i,j}z^j\mid 1\leq i\leq n\big)}$. This implies that 
    \[
    \sqrt{(f_{1,0},\ldots,f_{n,0},z)}=\sqrt{\bigg(\sum_{j=0}^d f_{i,j}z^j\mid 1\leq i\leq n\bigg)},
    \]
    which gives $\lcd(I)<n$ by the arguments of Remark \ref{rmk:ara-realized-by-local-cohomology}. (Even though Remark \ref{rmk:ara-realized-by-local-cohomology} is stated for polynomial rings, the arguments analogously hold for any $\KK$-algebra $R$.)
    For instance, if $d=e=t=1$ and $n=3$, we have
    \[R=\frac{\KK[a_1,a_2,a_3,b_1,b_2,b_3,c_1,c_2,c_3,d_1,d_2,d_3]}{(a_1c_1+a_2c_2+a_3c_3,\sum_{i=1}^3(a_id_i+b_ic_i)-1,b_1d_1+b_2d_2+b_3d_3)},\]
    and $I=(a_1,a_2,a_3)$. Observe that $a_1,a_2$ form a regular sequence in $R$, but $a_3$ is a zerodivisor on $R/(a_1,a_2)$. As $R$ is a complete intersection, $\height(I)=\grade(I)=2$ and as above, $\lcd(I)<3$. Hence, $\lcd(I)=2$.
\end{exa}

\section*{Acknowledgements}

We are grateful to Bhargav Bhatt, Jack Jeffries, Linquan Ma, and Bernd Ulrich for helpful discussions. A substantial portion of this work was completed when MB and VP were visiting the Institute for Advanced Study (IAS). They thank IAS for their warm hospitality and Linquan Ma for the invitation to visit.

MB was supported by NSF Grants DMS-2302430 and DMS-2100288, and by Simons Foundation Grant SFI-MPS-TSM-00012928.
UW was supported by NSF grant DMS-2100288, and by
Simons Foundation Collaboration Grant for Mathematicians \#580839 and SFI-MPS-TSM-00012928.

\subsection*{AI disclosure}
We used ChatGPT and Claude models to assist with literature searches and scanning the manuscript for typographical errors. We take full responsibility for the content and correctness of this article.

\bibliographystyle{amsalpha}
\bibliography{pfaff.bib}

@article {bmsmp,
    AUTHOR = {Batavia, Manav and Mohana Sundaram, Kesavan and Murray, Taylor
              and Pandey, Vaibhav},
     TITLE = {The arithmetic rank of the residual intersections of a
              complete intersection ideal},
   JOURNAL = {Int. Math. Res. Not. IMRN},
  FJOURNAL = {International Mathematics Research Notices. IMRN},
      YEAR = {2026},
    NUMBER = {10},
     PAGES = {Paper No. rnag091, 21},
      ISSN = {1073-7928,1687-0247},
   MRCLASS = {13C40},
  MRNUMBER = {5075936},
       DOI = {10.1093/imrn/rnag091},
       URL = {https://doi.org/10.1093/imrn/rnag091},
}

@article {RWW,
    AUTHOR = {Raicu, Claudiu and Weyman, Jerzy and Witt, Emily E.},
     TITLE = {Local cohomology with support in ideals of maximal minors and
              sub-maximal {P}faffians},
   JOURNAL = {Adv. Math.},
  FJOURNAL = {Advances in Mathematics},
    VOLUME = {250},
      YEAR = {2014},
     PAGES = {596--610},
      ISSN = {0001-8708,1090-2082},
   MRCLASS = {13D45 (13C40 14M12)},
  MRNUMBER = {3122178},
MRREVIEWER = {Kazem\ Khashyarmanesh},
       DOI = {10.1016/j.aim.2013.10.005},
       URL = {https://doi.org/10.1016/j.aim.2013.10.005},
}

@article {Barile,
    AUTHOR = {Barile, Margherita},
     TITLE = {Arithmetical ranks of ideals associated to symmetric and
              alternating matrices},
   JOURNAL = {J. Algebra},
  FJOURNAL = {Journal of Algebra},
    VOLUME = {176},
      YEAR = {1995},
    NUMBER = {1},
     PAGES = {59--82},
      ISSN = {0021-8693,1090-266X},
   MRCLASS = {13C40 (13D05 13F50)},
  MRNUMBER = {1345294},
MRREVIEWER = {Jos\'e\ F.\ Andrade},
       DOI = {10.1006/jabr.1995.1233},
       URL = {https://doi.org/10.1006/jabr.1995.1233},
}

@article {Witt2012,
    AUTHOR = {Witt, Emily E.},
     TITLE = {Local cohomology with support in ideals of maximal minors},
   JOURNAL = {Adv. Math.},
  FJOURNAL = {Advances in Mathematics},
    VOLUME = {231},
      YEAR = {2012},
    NUMBER = {3-4},
     PAGES = {1998--2012},
      ISSN = {0001-8708,1090-2082},
   MRCLASS = {13D45 (13C40)},
  MRNUMBER = {2964631},
MRREVIEWER = {Kazem\ Khashyarmanesh},
       DOI = {10.1016/j.aim.2012.07.001},
       URL = {https://doi.org/10.1016/j.aim.2012.07.001},
}

@incollection {HunekeUlrichGenericRI,
    AUTHOR = {Huneke, Craig and Ulrich, Bernd},
     TITLE = {Generic residual intersections},
 BOOKTITLE = {Commutative algebra ({S}alvador, 1988)},
    SERIES = {Lecture Notes in Math.},
    VOLUME = {1430},
     PAGES = {47--60},
 PUBLISHER = {Springer, Berlin},
      YEAR = {1990},
      ISBN = {3-540-52745-1},
   MRCLASS = {13C40 (13H10 14M12)},
  MRNUMBER = {1068323},
MRREVIEWER = {Matthew\ Miller},
       DOI = {10.1007/BFb0085536},
       URL = {https://doi.org/10.1007/BFb0085536},
}

@article {KustinUlrich,
    AUTHOR = {Kustin, Andrew R. and Ulrich, Bernd},
     TITLE = {If the socle fits},
   JOURNAL = {J. Algebra},
  FJOURNAL = {Journal of Algebra},
    VOLUME = {147},
      YEAR = {1992},
    NUMBER = {1},
     PAGES = {63--80},
      ISSN = {0021-8693,1090-266X},
   MRCLASS = {13C40 (13H10)},
  MRNUMBER = {1154674},
MRREVIEWER = {Aron\ Simis},
       DOI = {10.1016/0021-8693(92)90252-H},
       URL = {https://doi.org/10.1016/0021-8693(92)90252-H},
}

@article {BrunsSchwanzl,
    AUTHOR = {Bruns, Winfried and Schw\"anzl, Roland},
     TITLE = {The number of equations defining a determinantal variety},
   JOURNAL = {Bull. London Math. Soc.},
  FJOURNAL = {The Bulletin of the London Mathematical Society},
    VOLUME = {22},
      YEAR = {1990},
    NUMBER = {5},
     PAGES = {439--445},
      ISSN = {0024-6093,1469-2120},
   MRCLASS = {14M12},
  MRNUMBER = {1082012},
MRREVIEWER = {Gennady\ Lyubeznik},
       DOI = {10.1112/blms/22.5.439},
       URL = {https://doi.org/10.1112/blms/22.5.439},
}

@book {BrunsVetter,
    AUTHOR = {Bruns, Winfried and Vetter, Udo},
     TITLE = {Determinantal rings},
    SERIES = {Lecture Notes in Mathematics},
    VOLUME = {1327},
 PUBLISHER = {Springer-Verlag, Berlin},
      YEAR = {1988},
     PAGES = {viii+236},
      ISBN = {3-540-19468-1},
   MRCLASS = {13-02 (13C05 13H10 14M12)},
  MRNUMBER = {953963},
MRREVIEWER = {Piotr\ Pragacz},
       DOI = {10.1007/BFb0080378},
       URL = {https://doi.org/10.1007/BFb0080378},
}

@incollection {Kustin,
    AUTHOR = {Kustin, Andrew R.},
     TITLE = {Pfaffian identities, with applications to free resolutions,
              {DG}-algebras, and algebras with straightening law},
 BOOKTITLE = {Commutative algebra: syzygies, multiplicities, and birational
              algebra ({S}outh {H}adley, {MA}, 1992)},
    SERIES = {Contemp. Math.},
    VOLUME = {159},
     PAGES = {269--292},
 PUBLISHER = {Amer. Math. Soc., Providence, RI},
      YEAR = {1994},
      ISBN = {0-8218-5188-8},
   MRCLASS = {13D25 (13C40 15A69)},
  MRNUMBER = {1266186},
MRREVIEWER = {Jos\'e\ A.\ Hermida-Alonso},
       DOI = {10.1090/conm/159/01511},
       URL = {https://doi.org/10.1090/conm/159/01511},
}

@article {PeskineSzpiro,
    AUTHOR = {Peskine, Christian and Szpiro, Lucien},
     TITLE = {Dimension projective finie et cohomologie locale.
              {A}pplications \`a la d\'{e}monstration de conjectures de {M}.
              {A}uslander, {H}. {B}ass et {A}. {G}rothendieck},
   JOURNAL = {Inst. Hautes \'{E}tudes Sci. Publ. Math.},
  FJOURNAL = {Institut des Hautes \'{E}tudes Scientifiques. Publications
              Math\'{e}matiques},
    NUMBER = {42},
      YEAR = {1973},
     PAGES = {47--119},
      ISSN = {0073-8301},
   MRCLASS = {14B15 (13D05 13H10)},
  MRNUMBER = {374130},
MRREVIEWER = {Melvin Hochster},
       URL = {http://www.numdam.org/item?id=PMIHES_1973__42__47_0},
}

@book {Milne,
    AUTHOR = {Milne, James S.},
     TITLE = {{\'E}tale cohomology},
    SERIES = {Princeton Mathematical Series},
    VOLUME = {No. 33},
 PUBLISHER = {Princeton University Press, Princeton, NJ},
      YEAR = {1980},
     PAGES = {xiii+323},
      ISBN = {0-691-08238-3},
   MRCLASS = {14-02 (14F20 18F99)},
  MRNUMBER = {559531},
MRREVIEWER = {G.\ Horrocks},
}

@article{JPSW,
    author = {Jeffries, Jack and Pandey, Vaibhav and Singh, Anurag K. and Walther, Uli},
    title = {The arithmetic rank of determinantal nullcones},
    journal ={\url{https://arxiv.org/abs/2509.20470}},
    year = {2025}
}

@article {HunekeUlrichRI,
    AUTHOR = {Huneke, Craig and Ulrich, Bernd},
     TITLE = {Residual intersections},
   JOURNAL = {J. Reine Angew. Math.},
  FJOURNAL = {Journal f\"{u}r die Reine und Angewandte Mathematik. [Crelle's
              Journal]},
    VOLUME = {390},
      YEAR = {1988},
     PAGES = {1--20},
      ISSN = {0075-4102,1435-5345},
   MRCLASS = {13H10 (13C13 13D10 14B07)},
  MRNUMBER = {953673},
MRREVIEWER = {J\"{u}rgen\ Herzog},
       DOI = {10.1515/crll.1988.390.1},
       URL = {https://doi.org/10.1515/crll.1988.390.1},
}

@article {EisenbudEvans,
    AUTHOR = {Eisenbud, David and Evans, Jr., E. Graham},
     TITLE = {Every algebraic set in {$n$}-space is the intersection of
              {$n$}\ hypersurfaces},
   JOURNAL = {Invent. Math.},
  FJOURNAL = {Inventiones Mathematicae},
    VOLUME = {19},
      YEAR = {1973},
     PAGES = {107--112},
      ISSN = {0020-9910,1432-1297},
   MRCLASS = {14M10 (13B25)},
  MRNUMBER = {327783},
MRREVIEWER = {S.\ L.\ Kleiman},
       DOI = {10.1007/BF01418923},
       URL = {https://doi.org/10.1007/BF01418923},
}

@book {Silverman,
    AUTHOR = {Silverman, Joseph H.},
     TITLE = {The arithmetic of elliptic curves},
    SERIES = {Graduate Texts in Mathematics},
    VOLUME = {106},
   EDITION = {Second},
 PUBLISHER = {Springer, Dordrecht},
      YEAR = {2009},
     PAGES = {xx+513},
      ISBN = {978-0-387-09493-9},
   MRCLASS = {11-02 (11G05 11G20 14H52 14K15)},
  MRNUMBER = {2514094},
MRREVIEWER = {Vasil\cprime\ \=I.\ Andr\=\i\u ichuk},
       DOI = {10.1007/978-0-387-09494-6},
       URL = {https://doi.org/10.1007/978-0-387-09494-6},
}

@article {Artin-Nagata,
    AUTHOR = {Artin, Michael and Nagata, Masayoshi},
     TITLE = {Residual intersections in {C}ohen-{M}acaulay rings},
   JOURNAL = {J. Math. Kyoto Univ.},
  FJOURNAL = {Journal of Mathematics of Kyoto University},
    VOLUME = {12},
      YEAR = {1972},
     PAGES = {307--323},
      ISSN = {0023-608X},
   MRCLASS = {13H10},
  MRNUMBER = {301006},
MRREVIEWER = {Tadayuki\ Matsuoka},
       DOI = {10.1215/kjm/1250523522},
       URL = {https://doi.org/10.1215/kjm/1250523522},
}

@article {Kustin-Ulrich-Memoirs,
    AUTHOR = {Kustin, Andrew R. and Ulrich, Bernd},
     TITLE = {A family of complexes associated to an almost alternating map,
              with applications to residual intersections},
   JOURNAL = {Mem. Amer. Math. Soc.},
  FJOURNAL = {Memoirs of the American Mathematical Society},
    VOLUME = {95},
      YEAR = {1992},
    NUMBER = {461},
     PAGES = {iv+94},
      ISSN = {0065-9266,1947-6221},
   MRCLASS = {13D25 (13C14 13C40)},
  MRNUMBER = {1091668},
MRREVIEWER = {W.\ V.\ Vasconcelos},
       DOI = {10.1090/memo/0461},
       URL = {https://doi.org/10.1090/memo/0461},
}

@article {Pandey-Tarasova,
    AUTHOR = {Pandey, Vaibhav and Tarasova, Yevgeniya},
     TITLE = {Linkage and {$F$}-regularity of determinantal rings},
   JOURNAL = {Int. Math. Res. Not. IMRN},
  FJOURNAL = {International Mathematics Research Notices. IMRN},
      YEAR = {2024},
    NUMBER = {11},
     PAGES = {9323--9339},
      ISSN = {1073-7928,1687-0247},
   MRCLASS = {13A35 (13C40)},
  MRNUMBER = {4756115},
MRREVIEWER = {Mitra\ Koley},
       DOI = {10.1093/imrn/rnae040},
       URL = {https://doi.org/10.1093/imrn/rnae040},
}

@article {Smith,
    AUTHOR = {Smith, Karen E.},
     TITLE = {{$F$}-rational rings have rational singularities},
   JOURNAL = {Amer. J. Math.},
  FJOURNAL = {American Journal of Mathematics},
    VOLUME = {119},
      YEAR = {1997},
    NUMBER = {1},
     PAGES = {159--180},
      ISSN = {0002-9327,1080-6377},
   MRCLASS = {13A35 (13D45 13F40 14B05)},
  MRNUMBER = {1428062},
MRREVIEWER = {Ian\ M.\ Aberbach},
       URL =
        {http://muse.jhu.edu/journals/american_journal_of_mathematics/v119/119.1smith.pdf},
}

@article {Hochster-HunekeJAG,
    AUTHOR = {Hochster, Melvin and Huneke, Craig},
     TITLE = {Tight closure of parameter ideals and splitting in
              module-finite extensions},
   JOURNAL = {J. Algebraic Geom.},
  FJOURNAL = {Journal of Algebraic Geometry},
    VOLUME = {3},
      YEAR = {1994},
    NUMBER = {4},
     PAGES = {599--670},
      ISSN = {1056-3911,1534-7486},
   MRCLASS = {13A35 (13B21 13H10)},
  MRNUMBER = {1297848},
MRREVIEWER = {Ian\ M.\ Aberbach},
}

@article {Glassbrenner,
    AUTHOR = {Glassbrenner, Donna},
     TITLE = {Strong {$F$}-regularity in images of regular rings},
   JOURNAL = {Proc. Amer. Math. Soc.},
  FJOURNAL = {Proceedings of the American Mathematical Society},
    VOLUME = {124},
      YEAR = {1996},
    NUMBER = {2},
     PAGES = {345--353},
      ISSN = {0002-9939},
   MRCLASS = {13A35 (13H10)},
  MRNUMBER = {1291770},
MRREVIEWER = {Ian M. Aberbach},
       DOI = {10.1090/S0002-9939-96-03030-4},
       URL = {https://doi-org.ezproxy.lib.purdue.edu/10.1090/S0002-9939-96-03030-4},
}

@article {ChiriviMaffei,
    AUTHOR = {Chiriv\`i, Rocco and Maffei, Andrea},
     TITLE = {Pfaffians and shuffling relations for the spin module},
   JOURNAL = {Algebr. Represent. Theory},
  FJOURNAL = {Algebras and Representation Theory},
    VOLUME = {16},
      YEAR = {2013},
    NUMBER = {4},
     PAGES = {955--978},
      ISSN = {1386-923X,1572-9079},
   MRCLASS = {13C40 (13F50 14L35 14M12 14M15 58A17)},
  MRNUMBER = {3079788},
MRREVIEWER = {Toma\v z\ Ko\v sir},
       DOI = {10.1007/s10468-012-9341-7},
       URL = {https://doi.org/10.1007/s10468-012-9341-7},
}

@article {DeNegriSbarra,
    AUTHOR = {De Negri, Emanuela and Sbarra, Enrico},
     TITLE = {Gr\"obner bases of ideals cogenerated by {P}faffians},
   JOURNAL = {J. Pure Appl. Algebra},
  FJOURNAL = {Journal of Pure and Applied Algebra},
    VOLUME = {215},
      YEAR = {2011},
    NUMBER = {5},
     PAGES = {812--821},
      ISSN = {0022-4049,1873-1376},
   MRCLASS = {13C40 (13F50 13F55 13P10)},
  MRNUMBER = {2747221},
MRREVIEWER = {Marcel\ Morales},
       DOI = {10.1016/j.jpaa.2010.06.026},
       URL = {https://doi.org/10.1016/j.jpaa.2010.06.026},
}

@article {Chasles,
    AUTHOR = {Chasles, Michel},
     TITLE = {Construction des coniques qui satisfont \`a cinque conditions},
   JOURNAL = {C. R. Acad. Sci. Paris},
    VOLUME = {58},
      YEAR = {1864},
     PAGES = {297--308},
}

@article {HunekeRI,
    AUTHOR = {Huneke, Craig},
     TITLE = {Strongly {C}ohen-{M}acaulay schemes and residual
              intersections},
   JOURNAL = {Trans. Amer. Math. Soc.},
  FJOURNAL = {Transactions of the American Mathematical Society},
    VOLUME = {277},
      YEAR = {1983},
    NUMBER = {2},
     PAGES = {739--763},
      ISSN = {0002-9947,1088-6850},
   MRCLASS = {13H10 (14M05)},
  MRNUMBER = {694386},
MRREVIEWER = {Ulrich\ Orbanz},
       DOI = {10.2307/1999234},
       URL = {https://doi.org/10.2307/1999234},
}

@incollection {UlrichANproperties,
    AUTHOR = {Ulrich, Bernd},
     TITLE = {Artin-{N}agata properties and reductions of ideals},
 BOOKTITLE = {Commutative algebra: syzygies, multiplicities, and birational
              algebra ({S}outh {H}adley, {MA}, 1992)},
    SERIES = {Contemp. Math.},
    VOLUME = {159},
     PAGES = {373--400},
 PUBLISHER = {Amer. Math. Soc., Providence, RI},
      YEAR = {1994},
      ISBN = {0-8218-5188-8},
   MRCLASS = {13C40 (13A30 13H10)},
  MRNUMBER = {1266194},
MRREVIEWER = {Matthew\ Miller},
       DOI = {10.1090/conm/159/01519},
       URL = {https://doi.org/10.1090/conm/159/01519},
}

@article {CEU,
    AUTHOR = {Chardin, Marc and Eisenbud, David and Ulrich, Bernd},
     TITLE = {Hilbert functions, residual intersections, and residually
              {${\rm S}_2$} ideals},
   JOURNAL = {Compositio Math.},
  FJOURNAL = {Compositio Mathematica},
    VOLUME = {125},
      YEAR = {2001},
    NUMBER = {2},
     PAGES = {193--219},
      ISSN = {0010-437X,1570-5846},
   MRCLASS = {13D40 (13C40 13H15 14M06)},
  MRNUMBER = {1815393},
MRREVIEWER = {Philippe\ Gimenez},
       DOI = {10.1023/A:1002442111114},
       URL = {https://doi.org/10.1023/A:1002442111114},
}

@article {CNT,
    AUTHOR = {Chardin, Marc and Na\'eliton, Jos\'e{} and Tran, Quang Hoa},
     TITLE = {Cohen-{M}acaulayness and canonical module of residual
              intersections},
   JOURNAL = {Trans. Amer. Math. Soc.},
  FJOURNAL = {Transactions of the American Mathematical Society},
    VOLUME = {372},
      YEAR = {2019},
    NUMBER = {3},
     PAGES = {1601--1630},
      ISSN = {0002-9947,1088-6850},
   MRCLASS = {13C40 (13D02 13D40 13H10)},
  MRNUMBER = {3976571},
MRREVIEWER = {Yu\ Xie},
       DOI = {10.1090/tran/7607},
       URL = {https://doi.org/10.1090/tran/7607},
}

@article {EisenbudUlrichRI,
    AUTHOR = {Eisenbud, David and Ulrich, Bernd},
     TITLE = {Residual intersections of {$2\times n$} determinantal ideals},
   JOURNAL = {Algebr. Geom. Phys.},
  FJOURNAL = {Algebraic Geometry and Physics},
    VOLUME = {2},
      YEAR = {2025},
    NUMBER = {2},
     PAGES = {309--337},
      ISSN = {2998-100X,2998-1018},
   MRCLASS = {14M12 (13C40 13D02 13H10 13P20 14M06 14M10)},
  MRNUMBER = {5083940},
       DOI = {10.4310/agp.251009144743},
       URL = {https://doi.org/10.4310/agp.251009144743},
}

@article{BMMS,
    author = {Batavia, Manav and Maithani, Aryaman and Mohana Sundaram, Kesavan},
    title = {The arithmetic rank of nullcones of classical invariant rings},
    journal = {\url{https://arxiv.org/abs/2608.29445}},
    year = {2026}
}

@article{CowsikNori,
  author = {Cowsik, R. C. and Nori, Madhav V.},
  title = {Affine curves in characteristic {$p$} are set theoretic complete intersections},
  journal = {Invent. Math.},
  fjournal = {Inventiones Mathematicae},
  volume = {45},
  year = {1978},
  number = {2},
  pages = {111--114},
  issn = {0020-9910,1432-1297},
  mrclass = {14H99 (14M10)},
  mrnumber = {472835},
  mrreviewer = {Knud\ L\o nsted},
  doi = {10.1007/BF01390268},
  url = {https://doi.org/10.1007/BF01390268},
}

@article{HartshorneCI,
  author = {Hartshorne, Robin},
  title = {Complete intersections in characteristic {$p>0$}},
  journal = {Amer. J. Math.},
  fjournal = {American Journal of Mathematics},
  volume = {101},
  year = {1979},
  number = {2},
  pages = {380--383},
  issn = {0002-9327,1080-6377},
  mrclass = {14M10 (14H45)},
  mrnumber = {527998},
  mrreviewer = {Ezio\ Stagnaro},
  doi = {10.2307/2373984},
  url = {https://doi.org/10.2307/2373984},
}

@article {LyubeznikBAMS,
    AUTHOR = {Lyubeznik, Gennady},
     TITLE = {The number of equations needed to define an algebraic set},
   JOURNAL = {Bull. Amer. Math. Soc. (N.S.)},
  FJOURNAL = {American Mathematical Society. Bulletin. New Series},
    VOLUME = {19},
      YEAR = {1988},
    NUMBER = {1},
     PAGES = {273--276},
      ISSN = {0273-0979,1088-9485},
   MRCLASS = {14M10 (14C15 14C25)},
  MRNUMBER = {940486},
MRREVIEWER = {N.\ Mohan Kumar},
       DOI = {10.1090/S0273-0979-1988-15638-8},
       URL = {https://doi.org/10.1090/S0273-0979-1988-15638-8},
}

@article{BarileLyubeznik,
  author = {Barile, Margherita and Lyubeznik, Gennady},
  title = {Set-theoretic complete intersections in characteristic {$p$}},
  journal = {Proc. Amer. Math. Soc.},
  fjournal = {Proceedings of the American Mathematical Society},
  volume = {133},
  year = {2005},
  number = {11},
  pages = {3199--3209},
  issn = {0002-9939,1088-6826},
  mrclass = {14M10 (14F20 14M25)},
  mrnumber = {2160181},
  mrreviewer = {N.\ Mohan Kumar},
  doi = {10.1090/S0002-9939-05-07951-7},
  url = {https://doi.org/10.1090/S0002-9939-05-07951-7},
}

@article {BuchsbaumEisenbudAJM,
    AUTHOR = {Buchsbaum, David A. and Eisenbud, David},
     TITLE = {Algebra structures for finite free resolutions, and some
              structure theorems for ideals of codimension {$3$}},
   JOURNAL = {Amer. J. Math.},
  FJOURNAL = {American Journal of Mathematics},
    VOLUME = {99},
      YEAR = {1977},
    NUMBER = {3},
     PAGES = {447--485},
      ISSN = {0002-9327,1080-6377},
   MRCLASS = {13D15},
  MRNUMBER = {453723},
MRREVIEWER = {M.\ Nagata},
       DOI = {10.2307/2373926},
       URL = {https://doi.org/10.2307/2373926},
}

@article{HartshorneFlat,
  author = {Hartshorne, Robin},
  title = {A property of {$A$}-sequences},
  journal = {Bull. Soc. Math. France},
  fjournal = {Bulletin de la Soci\'et\'e{} Math\'ematique de France},
  volume = {94},
  year = {1966},
  pages = {61--65},
  issn = {0037-9484},
  mrclass = {13.95},
  mrnumber = {209279},
  mrreviewer = {L.\ J.\ Ratliff, Jr.},
  url = {http://www.numdam.org/item?id=BSMF_1966__94__61_0},
}

@article {LyubeznikAnnals,
    AUTHOR = {Lyubeznik, Gennady},
     TITLE = {{\'E}tale cohomological dimension and the topology of algebraic
              varieties},
   JOURNAL = {Ann. of Math. (2)},
  FJOURNAL = {Annals of Mathematics. Second Series},
    VOLUME = {137},
      YEAR = {1993},
    NUMBER = {1},
     PAGES = {71--128},
      ISSN = {0003-486X,1939-8980},
   MRCLASS = {14F20 (14B15 14M07)},
  MRNUMBER = {1200077},
MRREVIEWER = {R.\ T.\ Hoobler},
       DOI = {10.2307/2946619},
       URL = {https://doi.org/10.2307/2946619},
}

@article{Pandey,
AUTHOR = {Pandey,Vaibhav},
TITLE = {{$F$}-purity and the {$F$}-pure threshold as invariants of linkage},
JOURNAL ={\url{https://arxiv.org/abs/2406.05323}},
YEAR = {2024},
}

@article {Valla,
    AUTHOR = {Valla, Giuseppe},
     TITLE = {On determinantal ideals which are set-theoretic complete
              intersections},
   JOURNAL = {Compositio Math.},
  FJOURNAL = {Compositio Mathematica},
    VOLUME = {42},
      YEAR = {1980/81},
    NUMBER = {1},
     PAGES = {3--11},
      ISSN = {0010-437X,1570-5846},
   MRCLASS = {14M10 (13B25)},
  MRNUMBER = {594479},
       URL = {http://www.numdam.org/item?id=CM_1980__42_1_3_0},
}
\end{document}